\documentclass[10pt]{amsart}
\usepackage{amsfonts}
\usepackage{amssymb}
\usepackage{amsmath}
\newtheorem{theorem}{Theorem}[section]
\newtheorem{proposition}[theorem]{Proposition}
\newtheorem{lemma}[theorem]{Lemma}

\newtheorem{corollary}[theorem]{Corollary}

\def\R{{\mathbb R}}

\begin{document}

\title{Restriction for homogeneous polynomial surfaces in $\R^3$}
\author{A. Carbery, C. E. Kenig \and Sarah N. Ziesler}
\thanks{The second author was partially supported by NSF grant DMS-0456583 and DMS-0968472.}

\begin{abstract}
We prove an optimal restriction theorem for an arbitrary homogeneous polynomial hypersurface (of degree at least 2) in $\R^3,$ with affine curvature introduced as a mitigating factor.

\end{abstract}

\maketitle

\section{Introduction and statement of result}

If $S$ is a smooth $(n-1)$-dimensional submanifold in $\R^n (n\geq 3),$ $S_0$ is a compact subset with non-vanishing Gaussian curvature and $d\sigma$ is the induced Lebesgue measure, then the $L(p,q)$ Stein--Tomas restriction theorem (\cite{Stein},\cite{T})) says that, for all $f\in L^p(\R^n),$
$$\left(\int_{S_0}|\hat{f}|^qd\sigma \right)^{\frac1q}\leq C\|f\|_p,$$
for $1\leq p\leq \frac{2n+2}{n+3}, q\leq\left(\frac{n-1}{n+1}\right)p',$ where $\frac1p+\frac{1}{p'}=1.$ 
The key result is when $q=2,p=\frac{2n+2}{n+3}$ (the so-called $L^2$ restriction theorem) as the full range then follows by interpolation with the case $p=1.$

Our interest lies with analogues of the Stein--Tomas restriction
theorem for surfaces whose Gaussian curvature may vanish. In the case
where the surface is given by a graph, noncompact surfaces may also be
considered.  Other authors who have considered this problem include those of \cite{AS}, \cite{Bak}, \cite{BIT}, \cite{CZ}, \cite{CKZ2} and \cite{Oberlin}.
The first restriction theorems  for surfaces in $\R^n, n\geq 3,$ whose Gaussian curvature may vanish were given in \cite{KPV}.

We consider surfaces $\Gamma(\xi,\mu)=(\xi,\mu,\phi(\xi,\mu))$ in
$\R^3,$ where $\xi,\mu\in\R$ and $\phi :\R^2\rightarrow \R.$ To compensate for the possibly vanishing curvature we 
follow Sj\"{o}lin, \cite{S}, and also \cite{D} and \cite{KPV}, and
insert the mitigating factor $$|K_{\phi}
(\xi,\mu)|^{\frac14}=|[\det\text{Hess}\;\phi(\xi,\mu)]|^{\frac14}$$
into the left-hand side of the restriction inequality. This choice of
mitigating factor preserves the affine invariance of the restriction inequality. Moreover, with this mitigating factor, the restriction inequality is also invariant under reparametrization of the hypersurface. Because of this, we consider this to be the optimal choice of mitigating factor.

The standard approach to prove $L^2$ restriction theorems is via decay estimates for the Fourier transform of the measure supported on the surface. For example, Kenig, Ponce and Vega,  \cite{KPV}, showed that
the (mitigated) $L^2$ restriction property follows from a decay estimate of the form
$$\left|\int_{\Omega} e^{i(x\xi+y\mu+t\phi(\xi,\mu))}\psi(\xi,\mu)|K_{\phi}(\xi,\mu)|^{\frac12+i\beta}d\xi d\mu\right|\leq C (1+|\beta|)^2 |t|^{-1},$$
where $\phi:\Omega\rightarrow \R,$ $\psi\in C^{\infty}(\overline{\Omega})$ and $\psi$ and all its derivatives vanish on the boundary of $\Omega.$ Moreover, they used this to obtain a restriction theorem for polynomials of the form $p(\xi,\mu)=p_1(\xi)+p_2(\mu),$ where $p_1,p_2$ are arbitrary one variable polynomials of degree at least 2, as well as for homogeneous elliptic polynomials whose curvature is non-vanishing on the circle.  (We note that \cite{KPV} gives results for $\R^n, n\geq 3;$ we are only stating the $n=3$ case here.)
Unfortunately the desired decay estimates will not hold in general. Indeed they may fail even
for finite-type convex surfaces, as shown in \cite{CDMM}. In this paper we prove a restriction theorem for homogeneous polynomials; our basic approach is that of \cite{CKZ2}, adapted to the non-radial case.  Our method is special to $n=3,$ as will be seen.

\begin{theorem}\label{thm1}
Let $\Gamma(\xi,\mu)=(\xi,\mu,\phi(\xi,\mu))$ where $\phi(\xi,\mu)$ is
a polynomial, homogeneous of degree $d\geq 2.$
Then there is a constant $C$ (depending on $\phi$) such that
\begin{equation}\label{eq1}
\left(\int_{{\mathbb R}^2}|\hat{f}(\Gamma(\xi,\mu))|^2 |K_{\phi}(\xi,\mu)|^{\frac14}d\xi d\mu\right)^{\frac12}\leq C\|f\|_{\frac43},
\end{equation}
for all $f\in L^{\frac43}({\mathbb R}^3).$ 
\end{theorem}

\noindent
{\bf Remarks:} 1. We fix notation and take $\phi(\xi,\mu)=\sum_{i=0}^d a_i\xi^i\mu^{d-i}$ throughout.

2.  The constant $C$ depends on the polynomial $\phi$. It would be of great interest to
obtain a {\it uniform } result, i.e. where $C$ would depend only on
the degree of the polynomial.\medskip

Throughout what follows $C$ and $c$ will denote constants, possibly depending on $\phi,$ whose value may change from line to line. We will also write $K$ for $K_{\phi},$ to simplify notation.

A key ingredient in our proof is the following Littlewood-Paley theorem for homogenenous polynomials. The proof of this may be found in Section 3.

 We define, for $\lambda>1$ fixed,
\begin{equation}\label{eq***}
\psi^{\lambda}(t)=\begin{cases} 1 & 1\leq |t|\leq\lambda\\ 0 & |t|\leq \frac{1}{\lambda}\ \text{or}\ |t|\geq \lambda^2\end{cases}\quad\text{and} \quad  \psi^{\lambda}_k(t)=\psi^{\lambda}(\lambda^k t).
\end{equation}

Note that $1\leq \sum_k |\psi_k^\lambda (t)|^2,$ for $t\neq 0.$ When
$\lambda =2$ we omit the superscript $\lambda$ in the definitions of
$\psi^\lambda$ and $\psi^{\lambda}_k.$ In the proposition below the
symbol `$p$' is used in two different ways; this should not cause confusion.

\begin{proposition}\label{prop2}
a) Let $p : \R^2\rightarrow \R$ be a non-constant homogeneous polynomial and let $\lambda>1$ be fixed. 

Define
$$\widehat{(T_k f)}(\xi,\mu)=\psi^{\lambda}_k(p(\xi,\mu))\hat{f}(\xi,\mu).$$
Then there are absolute constants $c_1$ and $c_2$ (depending only on $\lambda$ and the degree of the polynomial) such that 
\begin{equation}\label{eq4}
c_1\leq \frac{\left\|\left(\sum_k |T_k f|^2\right)^\frac{1}{2}\right\|_p }{\|f\|_p}\leq c_2,
\end{equation}
for all $f\in L^p(\R^2),$ $1<p<\infty.$

b) Let $q:\R\longrightarrow \R$ be a one-variable non-constant
polynomial and let $\lambda>1$ be fixed.

Define 
$$\widehat{(\tilde{T}_k g)}(\xi,\mu) = \psi^{\lambda}_k (q(\mu/\xi))\hat{g}(\xi,\mu),$$
for $\hat{g}$ with support in $\{(\xi,\mu)\mid |\mu|\leq |\xi|\}$ Then there are absolute constants $c'_1$ and $c'_2$ (depending only on $\lambda$ and the degree of the polynomial) such that 
\begin{equation}\label{eq4b}
c'_1\leq \frac{\left\|\left(\sum_k |\tilde{T}_k g|^2\right)^\frac{1}{2}\right\|_p }{\|g\|_p}\leq c'_2,
\end{equation}
for all $g\in L^p(\R^2),$ $1<p<\infty,$ such that $\hat{g}$ has support in $\{(\xi,\mu)\mid |\mu|\leq |\xi|\}$. 
\end{proposition} 

\noindent
{\bf Remarks:} 1. Proposition \ref{prop2} is a strictly
two-dimensional result and moreover part a) 
holds only for homogeneous polynomials. This is readily seen as a consequence of results in
\cite{KT}.

2. Note that for any $f\in L^p(\R^2), \hat{g}(\xi,\mu)=\chi_{|\mu|\leq |\xi|}(\xi,\mu)\hat{f}(\xi,\mu)$ has $\|g\|_{p}\leq C\|f\|_{p}$ and \eqref{eq4b} applies to $g.$ (We use the usual notation that
$\chi_B$ is the characteristic function of the set $B.$)\medskip

To prove Theorem \ref{thm1} we first 
decompose into regions either near a zero of $K$ or away from the zeroes of
$K.$ As we will see in Lemmas \ref{lemmaZ1}, \ref{propZ2} and \ref{propZ3}, this allows us to assume either that $\frac{|K(1,\frac{\mu}{\xi})|}{\phi'(1,\frac{\mu}{\xi})^2}$ is bounded below (by an absolute constant $\epsilon>0$ depending on the polynomial $\phi$) or that the analogous property holds for a reparametrization of $\Gamma$. We use $\phi'(1,\eta) $ to denote the derivative in $\eta$ of $\phi(1,\eta).$ The parametrization invariance of the restriction property will be exploited  to deal with the latter case, once we have proved the result for the former.
Next we use annular Littlewood-Paley
theory, our polynomial Littlewood-Paley theory (Proposition \ref{prop2}), symmetry, and homogeneity to reduce the region of integration to one where
$|(\xi,\mu)|,$ $K(\xi,\mu)$ and $\phi_{\mu}(\xi,\mu)$ are essentially constant.  Then, to simplify matters even further, we show that our region of integration can be broken down into (open) intervals in $\xi$ and $\frac{\mu}{\xi}.$  The central idea in the proof is in Proposition \ref{prop2.5}, which then enables us  to complete the proof following the general strategy used in \cite{CKZ2}.  We use the fact that $\|f\|_{L^4}^2=\|f\overline{f}\|_{L^2}$ and then change variables, using 
Plancherel's theorem and Cauchy-Schwarz to pass from the estimate in Proposition \ref{prop2.5} to
the restriction result.  The assumption that $n=3$ is crucial as otherwise we will not be dealing with an $L^4$ norm.
 We will also need results from algebraic topology and real algebraic geometry, related to Bezout's theorem, that allow us to break down our regions of integration into intervals as well as to prove that we can reduce matters to a situation where the change of variables we use is effectively one-to-one. We would like to  thank Benson Farb for pointing out the useful reference \cite{H}. We would also like to thank the referee for his/her helpful comments on a previous version of this manuscript; these comments greatly helped us clarify the exposition.

\section{Proof of Theorem \ref{thm1}}\label{section2} 

We first observe that $K(1,\eta)$ is a polynomial of degree at most
$2(d-2),$ and we may assume that it is not identically zero;
 $K(1,\eta)$ therefore has at most $2(d-2)$ zeroes. Of course, if $d=2,$
$K$ will be constant and this step may be omitted. 
We now let $\omega_1,\omega_2,\dots,\omega_R$ be those zeroes of $K(1,\eta)$ for which the corresponding zeroes of $K$ on the unit circle
 $(\cos\theta_1,\sin \theta_1) ,(\cos\theta_2,\sin\theta_2),\dots, (\cos\theta_R,\sin\theta_R)$ satisfy $|\cos \theta_j|\geq |\sin\theta_j|$. So
$|\tan\theta_j|=|\omega_j|\leq 1,$ for each $j=1,\dots, R.$ 
We define, for each $j=1,\dots, R,$ 
$$\tilde{Z}_j=\{(\xi,\mu) \mid \xi\neq 0, \left|\tan^{-1}\left(\frac{\mu}{\xi}\right)-\theta_j\right|<\gamma_0\}$$
 where we choose $\gamma_0 $ so that each $\overline{\tilde{Z}_j}$ contains $(\cos \theta_j,\sin \theta_j)$ but contains no other zero of $K$ on the unit circle, and
 \begin{equation}\label{zeroes}
Z^*_j=\{(\xi,\mu)\in \tilde{Z}_j\mid |\mu|<|\xi|\}.
\end{equation}

We define also $$\tilde{Z}_0=\{(\xi,\mu)\mid \xi\neq 0,
\left|\tan^{-1}\left(\frac{\mu}{\xi}\right)-\theta_j\right|>\gamma_0,\forall
j=1,\dots , R\}$$
and 
\begin{equation}\label{zeroes2}
Z^*_0=\{(\xi,\mu)\in \tilde{Z}_0\mid |\mu|<|\xi|\}.
\end{equation}

We note that $\{(\xi,\mu)\mid |\mu|<|\xi|,\ (\xi,\mu)\not\in \bigcup_{j=0}^R Z^*_j\}$ is a set of measure zero in ${\R}^2.$

\begin{lemma}\label{lemmaZ1}
There is some $\epsilon_1>0,$ depending on $\phi,$ such that 
$$|K(1,\frac{\mu}{\xi})|\geq \epsilon_1 |\phi'(1,\frac{\mu}{\xi})|^2, \ \forall (\xi,\mu)\in Z^*_0.$$
\end{lemma}

\begin{proof}
If $(\xi,\mu)\in Z^*_0$ then for each $j=1,\dots, R$ we have
$|\tan^{-1}(\frac{\mu}{\xi})-\theta_j|>\gamma_0,$ where 
$\gamma_0$ is as defined before \eqref{zeroes}. Since $\frac12 \leq \frac{\tan^{-1}(\eta)-\theta_j}{\eta-\tan(\theta_j)}\leq 1$ for $|\eta|\leq 1,$ we have also
$|\frac{\mu}{\xi}-\tan(\theta_j)|> \gamma_0.$ Thus for $(\xi,\mu)\in Z^*_0$ we have that $\frac{\mu}{\xi}$ stays
away from the zeroes of $K(1,\frac{\mu}{\xi}).$ More explicitly, if 
$$K(1,\eta)=h(\eta) (\eta-\omega_1)^{k_1}\cdots (\eta-\omega_R)^{k_R},$$
where $h(\eta)\neq 0$ for $|\eta|\leq 1$ and
$k_1+k_2+\cdots + k_R\leq 2(d-2)$ (since $K$ has degree at most $2(d-2)$), then we have the bound
$$|K(1,\frac{\mu}{\xi})|\geq c_0\gamma_0^{k_1+k_2+\cdots + k_R},$$
with $c_0=\inf_{|\eta|\leq 1} |h(\eta)|.$
 Since $|\phi'(1,\frac{\mu}{\xi})|$ is bounded above for $|\mu|<|\xi|$
 (by $\sum_{i=1}^d|a_i|$, since
 $\phi(\xi,\mu)=\sum_{i=0}^d a_i\xi^i\mu^{d-i}$), this 
suffices to prove the result, with
 $\epsilon_1=$ \newline $ \frac{c_0\gamma_0^{k_1+k_2+\cdots + k_R}}{\left(\sum_{i=1}^d |a_i|\right)^2}.$
\end{proof}

\begin{lemma}\label{propZ2} 
Suppose that  
$\phi'(1,\omega_j)=0$ for some $j=1,\dots, R.$ 
Then 
there is some $\epsilon_2>0,$ depending on $\phi,$ such that 
$$|K(1,\frac{\mu}{\xi})|\geq \epsilon_2 |\phi'(1,\frac{\mu}{\xi})|^2, \ \forall (\xi,\mu)\in Z^*_j.$$ 
\end{lemma}

 
\begin{proof}
We consider first the case where $\phi(1,\omega_j)\neq 0.$ We let $1\leq m\leq d-1$ be the first integer such that $\phi^{(m+1)}(1,\omega_j)\neq 0$ and let $c=\frac{\phi^{(m+1)}(1,\omega_j)}{m!}.$ (We may assume that such an $m$ exists since otherwise $\phi'(1,\eta)\equiv 0,$ in which case the estimate is trivially true.) Then we have
\begin{eqnarray*}
\phi'(1,\eta) & = & c(\eta-\omega_j)^m + O(|\eta-\omega_j|^{m+1})\\
\phi''(1,\eta)& = & cm(\eta-\omega_j)^{m-1}+O(|\eta-\omega_j|^m)\\
\text{and}\quad \phi(1,\eta) & = & \phi(1,\omega_j) + \frac{c}{m+1}(\eta-\omega_j)^{m+1}+O(|\eta-\omega_j|^{m+2}).
\end{eqnarray*}
A straightforward calculation shows that 
$$K(1,\eta)  = d(d-1)\left\{ \phi(1,\eta)\phi''(1,\eta)-\frac{d-1}{d}\phi'(1,\eta)^2\right\}$$
and hence
$$K(1,\eta)=d(d-1) cm\phi(1,\omega_j)(\eta-\omega_j)^{m-1}+O(|\eta-\omega_j|^m).$$
Then for $|\eta-\omega_j|<\epsilon',$ $\epsilon'$ sufficiently small, we have
$$|K(1,\eta)|\geq d(d-1)\frac{cm}{2}|\phi(1,\omega_j)||\eta-\omega_j|^{m-1}\geq c^2|\eta-\omega_j|^{2m}\geq \frac{1}{10}\phi'(1,\eta)^2.$$

Now, for $(\xi,\mu)\in Z^*_j,$ $\eta=\frac{\mu}{\xi},$ we have $|\tan^{-1}\eta-\theta_j|<\gamma_0$ and hence $|\eta-\omega_j|<2\gamma_0,$ and then
if $2\gamma_0\leq \epsilon'$ we are done. If $\epsilon'<2\gamma_0 $
then we use the argument in Lemma \ref{lemmaZ1} to show that we obtain the desired bound on $\epsilon'\leq |\eta-\omega_j|<2\gamma_0.$ More precisely, we have
$$|K(1,\eta)|\geq c_0\gamma_0^{k_1+\cdots+k_{j-1}+k_{j+1}+\cdots + k_R} (\epsilon')^{k_j},\ \forall \eta \text{ such that } \epsilon'\leq |\eta-\omega_j|<2\gamma_0.$$
Hence we have 
$$|K(1,\frac{\mu}{\xi})|\geq \min (\frac{1}{10},
\frac{c_0\gamma_0^{k_1+\cdots k_{j-1}+k_{j+1}+\cdots + k_R}(\epsilon')^{k_j}}{\left(\sum_i |a_i|\right)^2}) \phi'(1,\frac{\mu}{\xi})^2,\ \forall (\xi,\mu) \in Z^*_j.$$
This completes the proof in this case, again using the fact that $|\phi'(1,\eta)|$ is bounded above for $|\eta|\leq 1.$

We now consider the case where $\phi(1,\omega_j)=0.$ A calculation similar to that in the first case gives us
\begin{eqnarray*}
\frac{1}{d(d-1)}|K(1,\eta)| & = & |c^2\left(\frac{m}{m+1}-\frac{d-1}{d}\right)(\eta-\omega_j)^{2m}|+O(|\eta-\omega_j|^{2m+1})\\ 
& \geq & \frac{c^2}{2}\left|\frac{m}{m+1}-\frac{d-1}{d}\right| |\eta-\omega_j|^{2m}\\
& \geq & \frac{1}{4}\left|\frac{m}{m+1}-\frac{d-1}{d}\right|\phi'(1,\eta)^2,
\end{eqnarray*}
for $|\eta-\omega_j|<\epsilon'',$ with $\epsilon''$ sufficiently small. 

Now if $1\leq m < d-1$ we have
$$\left|\frac{m}{m+1}-\frac{d-1}{d}\right|=\frac{|(d-1)-m|}{|(m+1)d|}\geq \frac{1}{d^2}$$
and hence $K(1,\eta)\geq \frac{1}{4d^2}\phi'(1,\eta)^2,$ for $|\eta-\omega_j|$ sufficiently small.
As in the previous case, we  need to consider whether $\epsilon''<2\gamma_0$ or $2\gamma_0\geq \epsilon''$; this is dealt with exactly as in the previous case.

If $m=d-1$ then $\phi'(1,\eta)=c(\eta-\omega_j)^{d-1},$
$\phi''(1,\eta)=c(d-1)(\eta-\omega_j)^{d-2}$ and
$\phi(1,\eta)=\frac{c}{d}(\eta-\omega_j)^d.$ This gives
$K(1,\eta)=0,\forall \eta,$  because of the polyomial nature of $K$, 
and hence $K\equiv 0$, by homogeneity. So this case does not occur. \end{proof}
 
\begin{lemma}\label{propZ3}
Suppose 
$\phi'(1,\omega_j)\neq 0.$ 
Then there is an affine transformation $T_j:{\mathbb R}^2\rightarrow {\mathbb R}^2$ such that if 
$$(\xi^0,\mu^0) := T_j(1,\omega_j) $$
and $\phi_j$ is defined by 
$$\phi_j(\xi,\mu)=\phi\circ T_j^{-1}(\xi,\mu),$$ 
then

(i) $\xi^0\neq 0$ and $\left|\frac{\mu^0}{\xi^0}\right|<\frac12$

(ii) $ c_j K_{\phi_j}(T_j(\xi,\mu))=K_{\phi}(\xi,\mu),$ with $c_j \neq 0$; in particular  $ K_{\phi_j}(1,\eta^0)=0,$
where $\eta^0=\frac{\mu^0}{\xi^0}$

(iii) $ \phi'_j(1,\eta^0)=0$
 
 and
 
 (iv) for all  $0<\epsilon_j<\gamma_0$ sufficiently small,  if
 $$Z^*_{j,\epsilon_j}=\{(\xi,\mu)\mid |\mu|<|\xi|,\left|\tan^{-1}\left(\frac{\mu}{\xi}\right)-\theta_j\right|<\epsilon_j\},$$
 then there is some $\tilde{\epsilon}_j>0$ with $\tilde{\epsilon}_j\rightarrow 0$ as $\epsilon_j\rightarrow 0,$ such that 
 $$T_j(Z^*_{j,\epsilon_j})\subset \{(\xi,\mu)\mid |\mu|<|\xi|,\left|\tan^{-1}\left(\frac{\mu}{\xi}\right)-\tan^{-1}\eta^0\right|<\tilde{\epsilon}_j\},$$ and  $K_{\phi_j}(1,\eta)$ has no zero other than $\eta^0$ in the closure of
 $$\{(\xi,\mu)\mid |\mu|<|\xi|,\left|\tan^{-1}\left(\frac{\mu}{\xi}\right)-\tan^{-1}\eta^0\right|<\tilde{\epsilon}_j\}.$$
 
\end{lemma}

\begin{proof}
We will choose unit vectors $u_0$ and $u_1$ such that $u_0\cdot u_1=0,$ and\newline 
$\nabla\phi(1,\omega_j)\cdot u_0 =0.$ Indeed, since $$K(1,\eta)=d(d-1)\left(\phi(1,\eta)\phi''(1,\eta)-\frac{d-1}{d}\phi'(1,\eta)^2\right),$$
and since
$K(1,\omega_j)=0,\ \phi'(1,\omega_j)\neq 0\Longrightarrow \phi(1,\omega_j)\neq 0.$
We also have $\nabla\phi(1,\omega_j)=(d\phi(1,\omega_j)-\omega_j\phi'(1,\omega_j),\phi'(1,\omega_j))\neq (0,0)$ and so we can define
 $$u_1= \frac{\nabla\phi(1,\omega_j)}{\|\nabla\phi(1,\omega_j)\|}=\frac{(d\phi(1,\omega_j)-\omega_j\phi'(1,\omega_j),\phi'(1,\omega_j))}{\|(d\phi(1,\omega_j)-\omega_j\phi'(1,\omega_j),\phi'(1,\omega_j))\|}$$
and 
$$u_0=\frac{(-\phi'(1,\omega_j),d\phi(1,\omega_j)-\omega_j\phi'(1,\omega_j))}{\|
(-\phi'(1,\omega_j),d\phi(1,\omega_j)-\omega_j\phi'(1,\omega_j))\|}.$$

We also note that 
\begin{equation}\label{star2}
(1,\omega_j)\cdot u_1=\frac{d\phi(1,\omega_j)}{\|(d\phi(1,\omega_j)-\omega_j\phi'(1,\omega_j),\phi'(1,\omega_j))\|}\neq 0.
\end{equation}

We now define $T_j$ by
$$T_j(\xi,\mu)=( (\xi,\mu)\cdot u_1, \frac{1}{\gamma_j}(\xi,\mu)\cdot u_0),$$
where $\gamma_j$ is to be chosen, $\gamma_j\neq 0.$ 
Clearly $T_j$ (a composition of a rotation with a scaling in the second variable) is an affine transformation. Moreover
$$T_j(1,\omega_j)=((1,\omega_j)\cdot u_1, \frac{1}{\gamma_j}(1,\omega_j)\cdot u_0)=(\xi^0,\mu^0)$$
where $\xi^0= (1,\omega_j)\cdot u_1\neq 0,$ by \eqref{star2}. 
Also,
$$\frac{\mu^0}{\xi^0}=\frac{1}{\gamma_j}\frac{[d\omega_j\phi(1,\omega_j)-(1+\omega_j^2)\phi'(1,\omega_j)]}{d\phi(1,\omega_j)}.$$
Since $|\omega_j|\leq 1$ and $\phi(1,\omega_j)\neq 0,$ we can choose $\gamma_j$ depending only on
$\phi,j$ so that $\left|\frac{\mu^0}{\xi^0}\right|<\frac{1}{2}.$ 
Thus we have shown (i).

Now $\phi_j(\xi,\mu)=\phi(T_j^{-1}(\xi,\mu))$ immediately gives
$$\text{Hess}\; \phi_j(\xi,\mu)=(T_j^{-1})^t\; \text{Hess}\;
\phi(T_j^{-1}(\xi,\mu))\; T_j^{-1}$$ and so
$c_jK_{\phi_j}(T_j(\xi,\mu))=K_{\phi}(\xi,\mu),$ where $c_j=(\det T_j)^2\neq 0.$ Moreover
\begin{eqnarray*}
c_j K_{\phi_j}(1,\eta^0)=K_{\phi}(T_j^{-1}(1,\eta^0)) & = & 
K_{\phi}\left(\frac{1}{\xi^0}T_j^{-1}(\xi^0,\mu^0)\right)\\
& = & \frac{1}{{(\xi^0)}^{2(d-2)}}K_{\phi}(1,\omega_j)=0,
\end{eqnarray*}
which gives (ii).

Note that $T_j(\xi u_1+\gamma_j\mu u_0)=\xi T_j(u_1)+\gamma_j \mu T_j(u_0)=\xi (1,0)+\gamma_j\mu (0,\frac{1}{\gamma_j})=(\xi,\mu)$ and hence
$$T_j^{-1}(\xi,\mu)=\xi u_1+\gamma_j \mu u_0.$$
Then  $\phi_j(\xi,\mu)=\phi(\xi u_1+\gamma_j \mu u_0)$. Also, as
$\eta^0=\frac{\mu^0}{\xi^0},$ we have
\begin{eqnarray*}
u_1+\gamma_j\eta^0 u_0 & = & u_1 + \gamma_j\left( \frac{\frac{1}{\gamma_j}(1,\omega_j)\cdot u_0}{(1,\omega_j)\cdot u_1}\right) u_0\\
& = & u_1 + \left(\frac{(1,\omega_j)\cdot u_0}{(1,\omega_j)\cdot u_1} \right) u_0\\
& = & \frac{1}{(1,\omega_j)\cdot u_1}\{ [(1,\omega_j)\cdot u_1]u_1+[(1,\omega_j)\cdot u_0] u_0\}\\
& = & \frac{1}{\xi^0}(1,\omega_j).
\end{eqnarray*}
Hence
\begin{eqnarray*}
\phi_j'(1,\eta^0) & = & \gamma_j u_0\cdot \nabla \phi(u_1+\gamma_j \eta^0 u_0)\\
& = & \gamma_j u_0\cdot \nabla\phi (\frac{1}{\xi^0}(1,\omega_j))\\
& = & \frac{\gamma_j}{(\xi^0)^{d-1}} u_0\cdot\nabla \phi(1,\omega_j)\\
& = & 0,
\end{eqnarray*}
which gives (iii).

For (iv), we first have that if $(\xi,\mu)\in Z^*_{j,\epsilon_j}$ and $\eta=\frac{\mu}{\xi}$, then
$|\eta-\omega_j|<2\epsilon_j.$   Also, we recall that $(1,\omega_j)\cdot u_1\neq 0$ and
$|(1,\eta)\cdot u_1-(1,\omega_j)\cdot u_1|\leq |\eta-\omega_j|,$ from which it follows that, for $\epsilon_j$ sufficiently small we have $|(1,\eta)\cdot u_1|\geq \frac12 |(1,\omega_j)\cdot u_1|.$ 
Then, for $\epsilon_j$ small,
\begin{eqnarray*}
& &\left|\frac{(\xi,\mu)\cdot u_0}{\gamma_j(\xi,\mu)\cdot u_1} -\frac{(1,\omega_j)\cdot u_0}{\gamma_j (1,\omega_j)\cdot u_1}\right| \\
& = & \left|\frac{[(1,\eta)\cdot u_0][ (1,\omega_j)\cdot
      u_1]-[(1,\omega_j)\cdot u_0][ (1,\eta)\cdot u_1]}{\gamma_j
    [(1,\eta)\cdot u_1][ (1,\omega_j)\cdot u_1]}\right|\\
& = & \frac{|\eta - \omega_j|}{|\gamma_j [(1,\eta)\cdot u_1][ (1,\omega_j)\cdot u_1]|}\\
& \leq & \frac{2|d^2\phi(1,\omega_j)^2-2d\omega_j\phi(1,\omega_j)\phi'(1,\omega_j)+(1+\omega_j^2)\phi'(1,\omega_j)^2|}{{|\gamma_j||((1,\omega_j)\cdot u_1)|^2}\|\nabla\phi(1, \omega_j)\|^2} |\eta-\omega_j|\\
& \leq & \frac{2|d^2\phi(1,\omega_j)^2-2d\omega_j\phi(1,\omega_j)\phi'(1,\omega_j)+(1+\omega_j^2)\phi'(1,\omega_j)^2|}{|\gamma_j| d^2\phi(1,\omega_j)^2}|\eta-\omega_j|,
\end{eqnarray*}
 which is small when $\epsilon_j$ is small. Thus, for $\epsilon_j$ suffficiently small we can find
 $\tilde{\epsilon}_j$ with $\tilde{\epsilon}_j \rightarrow 0$ as $\epsilon_j \rightarrow 0$ and
  $$\left|\tan^{-1}\left[\frac{(\xi,\mu)\cdot u_0}{\gamma_j(\xi,\mu)\cdot u_1}\right]-\tan^{-1}(\eta^0)\right|<\tilde{\epsilon}_j.$$
 Moreover, since $|\eta^0|\leq \frac12,$ for $\epsilon_j$ sufficiently small we have
 $\left|\frac{(\xi,\mu)\cdot u_0}{\gamma_j(\xi,\mu)\cdot u_1}\right|<1,$ which completes the proof of the inclusion in (iv). Finally, we have shown that $\frac{\mu^0}{\xi^0}$ is a zero of $K_{\phi_j}(1,\eta),$ and we know that there are at most $2(d-2)$ such zeroes. Also, $|\eta^0|\leq \frac12,$ and so it follows that for $\tilde{\epsilon}_j$ sufficiently small there are no other zeroes of $K_{\phi_j}(1,\frac{\mu}{\xi}) $ in the closure of 
 $\{(\xi,\mu)\mid |\mu|<|\xi|, |\tan^{-1}\left(\frac{\mu}{\xi}\right)-\tan^{-1}(\eta^0)|<\tilde{\epsilon}_j\}.$
 
 \end{proof}

\begin{lemma}\label{lemma2.3'}
Suppose that 
$\phi'(1,\omega_j)\neq 0.$  Then there is some $\epsilon_3>0,$ depending on $\phi,$ such that 
$$|K(1,\frac{\mu}{\xi})|\geq \epsilon_3 |\phi'(1,\frac{\mu}{\xi})|^2,\ \forall (\xi,\mu)\in  Z_j^*\setminus Z^*_{j,\epsilon_j}.$$
\end{lemma}

\begin{proof}
The proof is identical to that of Lemma \ref{lemmaZ1}.
\end{proof}

Lemma \ref{propZ3} and Lemma \ref{lemma2.3'} show that after a reparametrization we are in the same situation
as in Lemma \ref{propZ2}. By the parametrization invariance of the restriction
phenomenon, for each $j=1,\dots, R,$  we will be able to reduce
matters to the case where the hypotheses of Lemma \ref{propZ2} hold, as we will see.

We now define 
\begin{eqnarray*}
J_1 & = & \{j=1,\dots, R\mid 
\phi'(1,\omega_j)=0 
\}\cup \{0\}\\
J_2 & = & \{j=1,\dots, R\mid 
\phi'(1,\omega_j)\neq 0, 
\}.
\end{eqnarray*}

Then we define 
\begin{eqnarray*}
Z_0 & =   &  Z_0^*\cup \left(\bigcup_{j\in J_2} Z_j^*\setminus Z^*_{j,\epsilon_j}\right)\\
Z_j & = & Z_j^*, \text{ for }j\in J_1, j\neq 0\\
Z_j & = & Z^*_{j,\epsilon_j}, \text{ for }j\in J_2.
\end{eqnarray*}

Hence, by Lemmas \ref{lemmaZ1}, \ref{propZ2} and \ref{lemma2.3'}, for $j\in J_1,$ there exists some $\epsilon>0,$ such that 
\begin{equation}\label{lowerbound}
|K(1,\frac{\mu}{\xi})|\geq \epsilon |\phi'(1,\frac{\mu}{\xi})|^2,\ \forall (\xi,\mu)\in Z_j.
\end{equation}
Also, by Lemma \ref{propZ3}, for $j\in J_2,$ there is an affine transformation $T_j:\R^2\rightarrow \R^2$ such that 
$\phi_j=\phi\circ T_j^{-1}$ satisfies an analogous relationship. This relationship will be crucial later, in the proof of Proposition \ref{prop2.5}.

We define, for $j=0,\dots, R,$  
\begin{equation}\label{W}
\widehat{(W_j f)}(\xi,\mu)=\chi_{Z_j}(\xi,\mu)\hat{f}(\xi,\mu).
\end{equation}

For reasons that will be made clear later (in the proof of Lemma \ref{lemma4}) 
we define, for $\epsilon$ as given in \eqref{lowerbound}, 
\begin{eqnarray}\label{cepsilond}
c_{\epsilon,d} & = &  \frac{2^{21+4d}}{\epsilon}(d-1)^2\\
\delta & = & \min\{\frac{\epsilon}{2^{33+6d}(d-1)},\frac{d-1}{3d}\}\label{delta}\\
\lambda & = & (1+\delta)^{\frac{1}{3d}}\label{lambda}\\
\alpha & = & (1+\delta)^{\frac13}.\label{alpha}
\end{eqnarray}
We note that, since $0<\delta\leq \frac{d-1}{3d}<1,$ we have $1<\lambda<2$ and $1<\alpha<2.$ 
\bigskip

 We now continue by using annular  Littlewood-Paley theory to reduce the region of integration in \eqref{eq1}. So we recall (\cite{Stein2},\cite{Stein}) (using $\lambda>1$ as defined above in \eqref{lambda}) that if 
 $$\widehat{(S_k^\lambda f)}(\xi,\mu)=\psi^{\lambda^2}_{2k}(\xi^2+\mu^2)\hat{f}(\xi,\mu),$$ then there are absolute constants $c_3$ and $c_4$ (depending only on $\lambda$ and therefore only on $\epsilon$ and $d$)  such that 
\begin{equation}\label{eq5}
c_3\leq \frac{\left\|\left(\sum_k |S_k^\lambda f|^2\right)^\frac{1}{2}\right\|_p }{\|f\|_p}\leq c_4,
\end{equation}
for all $f\in L^p(\R^2),$ $1<p<\infty.$
A standard argument gives also absolute constants $c_5,c_6$ (depending only on $\epsilon$ and $d$) such that 
\begin{equation}\label{eq6} 
c_5\leq \frac{\left\|\left(\sum_k |S_k^\lambda S_k^\lambda f|^2\right)^{\frac12}\right\|_p}{\|f\|_p}\leq c_6,
\end{equation}
for all $f\in L^p (\R^2),$ $1<p<\infty.$ 

Now let $\alpha>1$ be as defined above in \eqref{alpha}, and define, for $\hat{f}$ supported on
$\{(\xi,\mu)\mid |\mu|\leq |\xi|\},$
\begin{eqnarray*}
\widehat{(P_k f)}(\xi,\mu) & = & \psi_k(K(1,\frac{\mu}{\xi}))\hat{f}(\xi,\mu)\\
\widehat{(Q_k^\alpha f)}(\xi,\mu) & = & \psi^\alpha_k (\phi'(1,\frac{\mu}{\xi}))\hat{f}(\xi,\mu).
\end{eqnarray*}
By Proposition \ref{prop2} part b) we have absolute constants $c_3',c_4'$ (depending only on d), $c_5', c_6'$ (depending only on d and $\epsilon$) such that 
\begin{equation}\label{eq5a}
c'_3\leq \frac{\left\|\left(\sum_k |P_k f|^2\right)^\frac{1}{2}\right\|_p }{\|f\|_p}\leq c'_4,
\end{equation}
and 
\begin{equation}\label{eq6a}
c'_5\leq \frac{\left\|\left(\sum_k |Q^\alpha_k f|^2\right)^\frac{1}{2}\right\|_p}{\|f\|_p}\leq c'_6,
\end{equation}
for all $f\in L^p (\R^2),\ 1<p<\infty.$ 

There are also absolute constants $c_3^{''}, c_4^{''}$ (depending only on d), $ c_5^{''},c_6^{''}$ (depending only on d and $\epsilon$) such that
\begin{equation}\label{eq5b}
c''_3\leq \frac{\left\|\left(\sum_{k}|P_k P_k f|^2\right)^\frac12\right\|_p}{\|f\|_p}\leq c''_4,
\end{equation}
 and 
\begin{equation}\label{eq6b}
c''_5\leq \frac{\left\|\left(\sum_k |Q^\alpha_k Q^\alpha_k f|^2\right)^{\frac12}\right\|_p}{\|f\|_p}\leq c''_6
\end{equation}
for all $f\in L^p (\R^2).$

We note that in the event that either $K(1,\eta) $ or $\phi'(1,\eta)$ is constant, the corresponding Littlewood-Paley theory is not needed and the argument simplifies.

We now define 

$$U(t)g(x,y)=\int_{ |\mu|< |\xi|} e^{it\phi(\xi,\mu)}e^{i(x\xi+y\mu)}\hat{g}(\xi,\mu)|K(\xi,\mu)|^{\frac18}d\xi d\mu.$$

Then, by duality,
$$\left(\int_{|\mu|<|\xi|}|\hat{f}(\Gamma(\xi,\mu))|^2|K(\xi,\mu)|^{\frac14}d\xi d\mu\right)^{\frac12}\leq C\|f\|_{\frac43},$$
holds for all $f\in L^{\frac43}(\R^3),$ if, and only if, 
$$\|U(t)g\|_{L^4_{xyt}}\leq C\|g\|_{L^2_{xy}}$$
holds for all $g\in L^2(\R^2).$ \medskip

\noindent
{\bf Remark:} Since there are at most $2(d-2) +1 $ values of $j,$ if we can show that 
\begin{equation}\label{eq200}
\|U(t)W_j g\|_{L^4_{xyt}}\leq C\|g\|_{L^2_{xy}}, \forall g\in L^2(\R^2),
\end{equation}
holds for all $j=1,\dots , R$ (with constant $C$ depending on $\phi$ and $j$) we will have proved Theorem \ref{thm1} with the integration on the left-hand-side restricted to $\{(\xi,\mu)\mid |\mu|<|\xi|\}.$ This suffices, since the integral over the region $\{(\xi,\mu)\mid |\xi|<|\mu|\}$ can be dealt with
in a similar way and  $\{(\xi,\mu)\mid |\mu|=|\xi|\}$ is a set of measure zero in $\R^2.$ (Here $W_j$ is as defined in \eqref{W}.) 

\begin{proposition}\label{prop1}
 
  If, for $j\in J_1$ fixed, 
\begin{equation}\label{LPreduction2}
\|U(t)S_0^\lambda Q^\alpha_l P_k W_j g\|_{L^4_{xyt}}\leq C\|g\|_{L^2_{xy}},\ \forall g\in L^2 (\R^2),
\end{equation}
 with constant $C$ depending on $\phi$ and $j,$  but independent of $k$ and $l,$ then 
 \begin{equation}\label{extforW}
 \|U(t) W_j g\|_{L^4_{xyt}}\leq C\|g\|_{L^2_{xy}},\ \forall g\in L^2(\R^2),
 \end{equation}
 with constant $C$ depending on $\phi$ and $j.$ 
\end{proposition}

\noindent

\begin{proof}

We assume that \eqref{LPreduction2} holds.  We will first use homogeneity to deduce that 
\begin{equation}\label{eq**}
\|U(t)S_m^\lambda Q^\alpha_l P_k W_j g\|_{L^4_{xyt}}\leq \tilde{C}\| g\|_{L^2_{xy}},
\end{equation}
for all $g\in L^2_{xy}(\R^2),$ with  $\tilde{C}$  independent of 
$m,k,$ and $l.$ 

We have
\begin{eqnarray*}
& & U(t)S_m^\lambda Q^\alpha_l P_k W_jg(x,y)  \\
&=  & \int_{|\mu|<|\xi|}e^{it\phi(\xi,\mu)+i(x\xi +y\mu)}\psi_{2m}^{\lambda^2}(\xi^2+\mu^2)
\psi_l^{\alpha}(\phi'(1,\frac{\mu}{\xi}))\psi_k(K(1,\frac{\mu}{\xi}))\chi_{Z_j}(\xi,\mu)\cdot\\
& & \qquad\qquad \hat{g}(\xi,\mu)|K(\xi,\mu)|^{\frac18}d\xi d\mu\\
&=& \int_{|\mu|<|\xi|} e^{it\lambda^{-md}\phi(\xi,\mu)}e^{i(\lambda^{-m}x\xi+\lambda^{-m}y\mu)}\psi^{\lambda^2}(\xi^2+\mu^2)\psi^\alpha_l(\phi'(1,\frac{\mu}{\xi}))\psi_k(K(1,\frac{\mu}{\xi}))\cdot \\
& & \qquad\qquad\chi_{Z_j}(\xi,\mu) \hat{g}(\lambda^{-m}\xi,\lambda^{-m}\mu) \lambda^{\frac{-m(d-2)}{4}-2m}|K(\xi,\mu)|^{\frac18}d\xi d\mu  \\
&  =& \lambda^{-\frac{md}{4}-\frac{3m}{2}}U(\lambda^{-md} t)S_0^\lambda Q_l^\alpha P_k W_j \tilde{g}(\lambda^{-m}x,\lambda^{-m}y),
\end{eqnarray*}
where $\hat{\tilde{g}}(\xi,\mu)=\hat{g}(\lambda^{-m}\xi,\lambda^{-m}\mu).$

So 
\begin{eqnarray*}
& & \|U(t)S_m^\lambda Q_l^\alpha P_k W_j g\|_{L^4_{xyt}} \\
& = & \lambda^{-\frac{md}{4}-\frac{3m}{2}}\left(\int |U(\lambda^{-md} t)S_0^\lambda Q_l^\alpha P_k W_j \tilde{g}(\lambda^{-m}x,\lambda^{-m}y)|^4 dxdydt \right)^{\frac14}\\
& = & \lambda^{-\frac{md}{4}-\frac{3m}{2}+\frac{md}{4}+\frac{2m}{4}}\|U(t)S_0^\lambda Q_l^\alpha P_k W_j \tilde{g}\|_{L^4_{xyt}}\\
& \leq & C\lambda^{-m}\|\tilde{g}\|_{L^2_{xy}},\quad\text{by \eqref{LPreduction2}}\\
& \leq & \tilde{C}\|g\|_{L^2_{xy}},
\end{eqnarray*}
which shows \eqref{eq**}. 

We now use the Littlewood-Paley theories to pass to \eqref{extforW}. We have, using 
\eqref{eq6}, \eqref{eq5b}, \eqref{eq6b},
\begin{eqnarray*}
\|U(t)W_j g\|_{L^4_{xyt}} & \leq & \frac{1}{c_5}\left\|\left(\sum_m |S_m^\lambda S_m^\lambda U(t)W_j g|^2\right)^{\frac12}\right\|_{L^4_{xyt}}\\
& \leq & \frac{1}{c_5}\left(\sum_m\|U(t)S_m^\lambda S_m^\lambda W_jg\|_{L^4_{xyt}}^2\right)^{\frac12}\\
& \leq & \frac{1}{c_5 c''_3}\left(\sum_m\left\|\left(\sum_k |P_k P_k U(t)S_m^\lambda S_m^\lambda W_j g|^2\right)^\frac{1}{2}\right\|_{L^4_{xyt}}^2\right)^{\frac12}\\
& \leq & \frac{1}{c_5 c''_3}\left(\sum_m \sum_k \|U(t)P_k P_k S_m^\lambda S_m ^\lambda W_j g\|_{L^4_{xyt}}^2\right)^{\frac12}\\
& \leq & \frac{1}{c_5 c''_3 c''_5}\left(\sum_m \sum_k\sum_l \|U(t)Q_l^\alpha Q_l^\alpha P_k P_k S_m^\lambda S_m^\lambda W_j g\|_{L^4_{xyt}}^2\right)^{\frac12}.
\end{eqnarray*}

Next, using \eqref{eq**} followed by \eqref{eq5}, \eqref{eq5a} and \eqref{eq6a}, this gives
\begin{eqnarray*}
\|U(t)W_j g\|_{L^4_{xyt}} & \leq & \frac{\tilde{C}}{c_5 c''_3 c''_5}\left(\sum_m\sum_k\sum_l \|S_m^\lambda Q_l^\alpha P_k  g\|_{L^2_{xy}}^2\right)^{\frac12}\\
& = & \frac{\tilde{C}}{c_5 c''_3 c''_5}\left(\sum_l\sum_k\left\|\left(\sum_m |S_m^\lambda Q_l^\alpha P_k g|^2\right)^{\frac12}\right\|_{L^2_{xy}}^2\right)^{\frac12}\\
& \leq & \frac{\tilde{C} c_4}{c_5 c''_3 c''_5}\left(\sum_l\sum_k\|Q_l^\alpha P_k g\|_{L^2_{xy}}^2\right)^{\frac12}\\
& = & \frac{\tilde{C} c_4}{c_5 c''_3 c''_5}\left(\sum_l\left\|\left(\sum_k|Q_l^\alpha P_k g|^2\right)^{\frac12}\right\|_{L^2_{xy}}^2\right)^{\frac12}\\
& \leq & \frac{\tilde{C} c_4 c_4'}{c_5 c_3'' c_5''}\left(\sum_l\|Q_l^\alpha g\|_{L^2_{xy}}^2\right)^{\frac12}\\
& = & \frac{\tilde{C} c_4 c_4'}{c_5 c''_3 c''_5}\left\|\left(\sum_l |Q_l^\alpha g|^2\right)^{\frac12}\right\|_{L^2_{xy}}\\
& \leq & \frac{\tilde{C} c_4 c_4' c_6'}{c_5 c_3'' c_5''}\|g\|_{L^2_{xy}}.
\end{eqnarray*}
Thus we have shown \eqref{extforW}. 

\end{proof}

By Proposition \ref{prop1} the region of integration that we need to consider is now restricted to where $|(\xi,\mu)|$, $|K(1,\frac{\mu}{\xi})|$ , and $|\phi'(1,\frac{\mu}{\xi})|$ are essentially constant. However it simplifies calculations considerably if we can further reduce to the situation where $\xi$ and  $\frac{\mu}{\xi}$ each lie in an interval. Since $\phi$ is  a polynomial we are able to reduce to this case as follows.

For ease of notation we define
$$\Psi_{0lk}(\xi,\mu)=\psi^{\lambda^2}(\xi^2+\mu^2)\psi^{\alpha}_l (\phi'(1,\frac{\mu}{\xi}))\psi_k (K(1,\frac{\mu}{\xi})).$$
First we note that $ |\mu|< |\xi|,  \Psi_{0lk}(\xi,\mu)\neq 0\Longrightarrow $
$$\frac{1}{\sqrt{2}\lambda} < \frac{1}{\sqrt{2}}|(\xi,\mu)|<|\xi|\leq |(\xi,\mu)|< \lambda^{2}$$ 
and so 
$\xi\in (\frac{1}{\sqrt{2}\lambda},\lambda^2) \cup (- \lambda^{2}, -\frac{1}{\sqrt{2}\lambda}).$
Now $(\frac{1}{\sqrt{2}\lambda},\lambda^2)$  can be written as a finite union of intervals, $(\frac{1}{\sqrt{2}\lambda},\lambda^2)=\bigcup_{p=1}^{P} {\mathcal I}_{p}$ such that for each $p,$
$t_1,t_2\in {\mathcal I}_p\Longrightarrow \frac{1}{\lambda^3}\leq \frac{t_1}{t_2}\leq \lambda^3.$ 
This is easily seen by taking the ${\mathcal I} _p$ to be equally spaced intervals of length $\delta=\frac{1}{2} \frac{(\lambda^3-1)}{\sqrt{2}\lambda}.$ Then the number of intervals, $P,$ clearly depends only on $\lambda,$ which by definition depends only on $d$ and $\epsilon.$ (Recall \eqref{lambda},\eqref{delta}.) 
Then
$${\mathcal I}_p = \begin{cases} (\frac{1}{\sqrt{2}\lambda}+(p-1)\delta,\frac{1}{\sqrt{2}\lambda}+p\delta ] & p=1,2,\dots,P-1\\
(\frac{1}{\sqrt{2}\lambda}+(P-1)\delta,\lambda^2) & p=P.
\end{cases}$$
Thus, for $p=1,\dots, P-1,$  $t_1,t_2\in {\mathcal I}_p=(\frac{1}{\sqrt{2}\lambda}+ (p-1)\delta, \frac{1}{\sqrt{2}\lambda}+p\delta]
\Longrightarrow t_i=\frac{1}{\sqrt{2}\lambda}+(p-1+\theta_i)\delta, $ for some $\theta_i\in (0,1],\ i=1,2.$
Then
$$\frac{t_1}{t_2}  =  \frac{\frac{1}{\sqrt{2}\lambda}+(p-1)\delta+\theta_1\delta}{\frac{1}{\sqrt{2}\lambda}+(p-1)\delta+\theta_2\delta}$$
and so
\begin{eqnarray*}
\frac{1}
{1+\frac{\theta_2\delta}{\frac{1}{\sqrt{2}\lambda}+(p-1)\delta}} 
 \leq & \frac{t_1}{t_2} &
\leq 1 + \frac{\theta_1\delta}{\frac{1}{\sqrt{2}\lambda}+(p-1)\delta};\\
\text{thus}\quad \frac{1}{\lambda^3}\leq \frac{1}{1+\sqrt{2}\lambda\delta} \leq &\frac{t_1}{t_2} & \leq 1+ \sqrt{2}\lambda\delta\leq \lambda^3.
\end{eqnarray*}
The argument for ${\mathcal I}_P$ is similar.  The same argument gives a decomposition \newline $(-\lambda^2,-\frac{1}{\sqrt{2}\lambda})=\bigcup_{p=1}^{P}\tilde{\mathcal I}_p$ where $t\in \tilde{\mathcal I}_p\Longleftrightarrow -t\in {\mathcal I}_p.$

 We also have $|\mu|<|\xi|,\ \Psi_{0lk}(\xi,\mu)\neq 0\Longrightarrow  2^{-k-1} < |K(1,\frac{\mu}{\xi})|< {2^{-k+2}}$.
We now let $E_k=\{\eta \mid 2^{-k-1}< |K(1,\eta)|< 2^{-k+2} \}.$
Since $E_k$ is open and $K$ is polynomial, $E_k$  can be written as a finite union of disjoint open intervals,
$E_k=\bigcup_{m=1}^{M(k)}I_{k,m},$ where each component interval $I_{k,m}$ has endpoints
given by $|K(1,\eta)|= 2^{-k-1}$ or $|K(1,\eta)|=2^{-k+2}.$ Moreover $K(1,\eta)$ may be taken to be  single-signed on each interval. We also note that since $K(1,\cdot)$ is polynomial, $M(k)$ will be bounded independently of $k$; $M(k)\leq M(d)$ for some constant
$M$ depending only on the degree $d$ of $\phi.$

We can argue similarly for $\phi'(1,\frac{\mu}{\xi}).$ We have $|\mu|< |\xi|,\Psi_{0lk}(\xi,\mu)\neq 0\Longrightarrow $
$\alpha^{-l-1}< |\phi'(1,\frac{\mu}{\xi})|< \alpha^{-l+2}.$
Then if we let 
$F_l=\{\eta \mid \alpha^{-l-1}<|\phi'(1,\eta)|<\alpha^{-l+2}\}$
we can write $F_l$ as a finite union of disjoint intervals, $F_l=\bigcup_{n=1}^{N}I_{l,n},$ with
$N$ bounded and the bound depends only on $d.$ The $I_{l,n}$ have endpoints given by $|\phi'(1,\eta)|=\alpha^{-l-1}$ or $|\phi'(1,\eta)|=\alpha^{-l+2}$ and we may assume that $\phi'(1,\eta)$ is single-signed on each
interval. 

Thus we have
$$\Psi_{0lk}(\xi,\mu)=\Psi_{0lk}(\xi,\mu) \sum_{n=1}^N\sum_{m=1}^M \sum_{p=1}^P \chi_{I_{k,m}}(\frac{\mu}{\xi})\chi_{I_{l,n}}(\frac{\mu}{\xi})[\chi_{\overset{\:\circ}{\mathcal I}_p}(\xi)+\chi_{\overset{\:\circ}{\tilde{\mathcal I}}_p}(\xi)]$$
 almost everywhere. Here $N$ and $M $ are bounded with bound depending only on $d,$ and $P$ depends on $d$ and $\epsilon.$  
  (We note that $\{(\xi,\mu)\mid \xi\ \text{is an endpoint of }{\mathcal I}_p\cup\tilde{\mathcal I}_p\}$ is a set of measure zero in $\R^2.$)

We now define, for $m=1,\dots, M, n=1,\dots, N, p=1,\dots,P,$
\begin{eqnarray*}
R^1_{m,n,p}(k,l) & = & \{ (\xi,\mu)\mid \left|\frac{\mu}{\xi}\right|<1,\ \xi\in \overset{\:\circ}{\mathcal I}_p, \frac{\mu}{\xi}\in I_{k,m}\cap I_{l,n}\}\\
R^2_{m,n,p}(k,l) & = & \{ (\xi,\mu)\mid \left|\frac{\mu}{\xi}\right|<1,\ \xi\in \overset{\:\circ}{\tilde{\mathcal I}}_p,\frac{\mu}{\xi}\in I_{k,m}\cap I_{l,n}\}.
\end{eqnarray*}
and
\begin{eqnarray*}
S^1_{m,n,p,j}(k,l) &= & R^1_{m,n,p}(k,l)\cap Z_j,\ j=0,\dots, R\\
S^2_{m,n,p,j}(k,l) & = & R^2_{m,n,p}(k,l)\cap Z_j,\ j=0,\dots, R.
\end{eqnarray*}
We also define, for $m,n,p,j$ fixed, $s=1,2,$ $a>0,$ 
$$
A^s_{m,n,p,j}(k,l,a)=\{(\xi_1,\xi_2, \eta_1,\eta_2)\mid (\xi_1,\xi_1\eta_1), (\xi_2,\xi_2\eta_2)\in S^s_{m,n,p,j}(k,l), \eta_2>\eta_1+a\}.$$

\begin{lemma}\label{lemma2.2}
Suppose that for $m,n,p,j$ fixed, $j\in J_1,$ $s=1,2,$ $a>0,$  we have
\begin{eqnarray}
& & \bigg \|\int_{A_{m,n,p,j}^{s}(k,l,a)} e^{it(\xi_1^d\phi(1,\eta_1)+\xi_2^d\phi(1,\eta_2))}
e^{i[x(\xi_1+\xi_2)+y(\xi_1\eta_1+\xi_2\eta_2)]}|\xi_1\xi_2|^{\frac{d-2}{4}}\cdot \nonumber\\
& & \quad \left.|K(1,\eta_1)K(1,\eta_2)|^{\frac18}\Psi_{0lk}(\xi_1,\xi_1\eta_1)\Psi_{0lk}(\xi_2,\xi_2\eta_2)\hat{g}(\xi_1,\xi_1\eta_1)\hat{g}(\xi_2,\xi_2\eta_2)\right.\nonumber \\
& &  \quad \xi_1\xi_2 d\xi_1 d\xi_2 d\eta_1 d\eta_2\bigg\|_{L^2_{xyt}}\nonumber\\
& &\leq  C\|g\|_{L^2_{xy}}^2,\label{eq*3}
\end{eqnarray}
with $C$ independent of $k,l,$ and $a.$ Then
$$\|U(t)S^\lambda_0 Q^\alpha_l P_k  W_j g\|_{L^4_{xyt}}\leq C\| g\|_{L^2_{xy}},$$
with $C$ independent of $k$ and $l$.
\end{lemma} 

\begin{proof}
First we note that \eqref{eq*3} implies that
\begin{eqnarray*}
& & \|
\int_{\substack{(\xi_1,\xi_1\eta_1)\in S^s_{m,n,p,j}(k,l)\\
(\xi_2,\xi_2\eta_2)\in S^s_{m,n,p,j}(k,l)\\
\eta_2\geq \eta_1}}
e^{it(\xi_1^d\phi(1,\eta_1)+\xi_2^d\phi(1,\eta_2))}
e^{i[x(\xi_1+\xi_2)+y(\xi_1\eta_1+\xi_2\eta_2)]}|\xi_1\xi_2|^{\frac{d-2}{4}}\cdot \\
& & \quad \left.|K(1,\eta_1)K(1,\eta_2)|^{\frac18}\Psi_{0lk}(\xi_1,\xi_1\eta_1)\Psi_{0lk}(\xi_2,\xi_2\eta_2)\widehat{W_j g}(\xi_1,\xi_1\eta_1)\widehat{W_j g}(\xi_2,\xi_2\eta_2)\right. \\
& &  \quad \xi_1\xi_2 d\xi_1 d\xi_2 d\eta_1 d\eta_2\|_{L^2_{xyt}}\\
& &\leq  C\|W_j g\|_{L^2_{xy}}^2\\
& & \leq C\|g\|_{L^2_{xy}},
\end{eqnarray*}
with $C$ independent of $k$ and $l.$ Then, by symmetry, we obtain
\begin{eqnarray*}
& & \|
\int_{\substack{(\xi_1,\xi_1\eta_1)\in S^s_{m,n,p,j}(k,l)\\
(\xi_2,\xi_2\eta_2)\in S^s_{m,n,p,j}(k,l)}}
e^{it(\xi_1^d\phi(1,\eta_1)+\xi_2^d\phi(1,\eta_2))}
e^{i[x(\xi_1+\xi_2)+y(\xi_1\eta_1+\xi_2\eta_2)]}|\xi_1\xi_2|^{\frac{d-2}{4}}\cdot \\
& & \quad \left.|K(1,\eta_1)K(1,\eta_2)|^{\frac18}\Psi_{0lk}(\xi_1,\xi_1\eta_1)\Psi_{0lk}(\xi_2,\xi_2\eta_2)\widehat{W_j g}(\xi_1,\xi_1\eta_1)\widehat{W_j g}(\xi_2,\xi_2\eta_2)\right. \\
& &  \quad \xi_1\xi_2 d\xi_1 d\xi_2 d\eta_1 d\eta_2\|_{L^2_{xyt}}\\
& &\leq  C\|g\|_{L^2_{xy}}^2,
\end{eqnarray*}
with $C$ independent of $k$ and $l.$ 

Then, using the fact that $\|h\|_4=\|h\overline{h}\|_2^{\frac12},$ followed by the change of variables $\mu=\xi\eta,$ we find that
\begin{eqnarray*}
& & \left\|\int_{R^s_{m,n,p}(k,l)} e^{it\phi(\xi,\mu)}e^{i(x\xi+y\mu)}|K(\xi,\mu)|^{\frac18}\Psi_{0lk}(\xi,\mu)\widehat{W_j g}(\xi,\mu)d\xi d\mu\right\|_{L^4_{xyt}}\\
& & \leq C\|g\|_{L^2_{xy}},
\end{eqnarray*}
with $C$ independent of $k $ and $l.$ 

So
\begin{eqnarray*}
& & \|U(t)S_0^{\lambda} Q^\alpha_l P_k W_j g\|_{L^4_{xyt}} \\
&=  & \left\|\int_{|\mu|<|\xi|}e^{it\phi(\xi,\mu)}e^{i(x\xi+y\mu)}|K(\xi,\mu)|^{\frac18}\Psi_{0lk}(\xi,\mu)\widehat{W_j g}(\xi,\mu)d\xi d\mu\right\|_{L^4_{xyt}} \\
& \leq  & \sum_{m,n,p} \left\|\int_{R^1_{m,n,p}(k,l)}e^{it\phi(\xi,\mu)}e^{i(x\xi+y\mu)}|K(\xi,\mu)|^{\frac18}\Psi_{0lk}(\xi,\mu)\widehat{W_j g}(\xi,\mu)d\xi d\mu\right\|_{L^4_{xyt}} + \\
&&     \sum_{m,n,p} \left\|\int_{R^2_{m,n,p}(k,l)}e^{it\phi(\xi,\mu)}e^{i(x\xi+y\mu)}|K(\xi,\mu)|^{\frac18}\Psi_{0lk}(\xi,\mu)\widehat{W_j g}(\xi,\mu)d\xi d\mu\right\|_{L^4_{xyt}}.
\end{eqnarray*}
Since the sums in $m,n$ and $p$ are finite, with the number of terms depending only on 
$d$ and $\epsilon,$ we have
$$\|U(t)S_0^{\lambda} Q^\alpha P_k W_j g\|_{L^4_{xyt}}\leq C\|g\|_{L^2_{xy}},\ \forall g\in L^2(\R^2),$$
with $C$ independent of $k$ and $l.$ 

\end{proof}

We recall that  for $ j\in J_1,$ we have \eqref{lowerbound};  we will show that in this case \eqref{eq*3} holds. For $j\in J_2,$ Lemma \ref{propZ3} applies and we will show that an estimate analogous to \eqref{eq*3} holds after we apply the appropriate affine transformation $T_j:\R^2\rightarrow \R^2$ to $\phi.$ 

We consider first $j\in J_1.$ We also recall that $K(1,\eta)$ is single-signed on each $I_{k,m}.$ We suppose first that we are considering $A_{m,n,p,j}^s(k,l,a)$ with $m$ such that $K(1,\eta)$ is positive on $I_{k,m}.$ 
Then, in order to bound the left-hand-side in \eqref{eq*3} we make the following change of variables:
 \begin{eqnarray}
 u& = & \xi_1+\xi_2\nonumber \\
 v& = & \xi_1\eta_1+\xi_2\eta_2\nonumber \\
 w & = & \xi_1^d\phi(1,\eta_1)+\xi_2^d\phi(1,\eta_2)\label{eq*10}\\
 z & = & \xi_2^d \phi(1,\eta_2)-\xi_1^d\phi(1,\eta_1)\nonumber.
 \end{eqnarray}
If $K(1,\eta)$ is negative on $I_{k,m}$ then we use instead $z=\xi_1^d\phi(1,\eta_1)-\xi_2^d\phi(1,\eta_2).$ Since the proof follows in exactly the same way we omit the details of this case.

Unfortunately it is not clear that this is a one-to-one map. However we are able to show that it is at most many-to-one on our region of integration, and this will suffice.
We define
\begin{equation}\label{eq*11}
V(\xi_1,\xi_2,\eta_1,\eta_2)=(u,v,w,z).
\end{equation}
Then 
\begin{eqnarray*}
JV(\xi_1,\xi_2,\eta_1,\eta_2) & = & \frac{\partial (u,v,w,z)}{\partial (\xi_1,\xi_2,\eta_1,\eta_2)}\\
& = & 2(\eta_2-\eta_1)(\xi_1\xi_2)^d \phi'(1,\eta_1)\phi'(1,\eta_2)\frac{G(\eta_2)-G(\eta_1)}{\eta_2-\eta_1},
\end{eqnarray*}
where $G(\eta)=\eta-d\frac{\phi(1,\eta)}{\phi'(1,\eta)}.$
We note that $\frac{G(\eta_2)-G(\eta_1)}{\eta_2-\eta_1}=G'(\eta_3),$ for some $\eta_3$ between
$\eta_1$ and $\eta_2.$ Since $G'(\eta)=\frac{K(1,\eta)}{(d-1)\phi'(1,\eta)^2}$ and $\xi_1,\xi_2$ have the same sign, as do $\phi'(1,\eta_1)$ and $\phi'(1,\eta_2),$  we note that 
$JV(\xi_1,\xi_2,\eta_1,\eta_2)>0, \forall (\xi_1,\xi_2,\eta_1,\eta_2)\in A_{m,n,p,j}^s(k,l,a),$ since $a>0.$ In fact, we can find bounds on $JV $ in $A_{m,n,p,j}^s(k,l,a),$ as follows. We have  
 \begin{eqnarray*}
 |JV(\xi_1,\xi_2,\eta_1,\eta_2)|& = & |2(\eta_2-\eta_1)(\xi_1\xi_2)^d\phi'(1,\eta_1)\phi'(1,\eta_2)G'(\eta_3)|\\
 & = & \left|\frac{2}{d-1}(\eta_2-\eta_1)(\xi_1\xi_2)^d \frac{\phi'(1,\eta_1)\phi'(1,\eta_2)}{\phi'(1,\eta_3)^2}K(1,\eta_3)\right|,
 \end{eqnarray*}
 for some $\eta_1<\eta_3<\eta_2.$ Now, $(\xi_1,\xi_2,\eta_1,\eta_2)\in A^s_{m,n,p,j}(k,l,a)\Longrightarrow   \eta_1,\eta_2\in I_{l,n}\Longrightarrow $  \newline $  \alpha^{-l-1}<|\phi'(1,\eta_i)|<\alpha^{-l+2},$ for $i=1,2.$ Then
 we have the same inequality for $\eta_3, $ since $\eta_1<\eta_3<\eta_2$ and $I_{l,n}$ is an interval. Hence 
 $\frac{1}{\alpha^6}<\frac{\phi'(1,\eta_1)\phi'(1,\eta_2)}{\phi'(1,\eta_3)^2}<\alpha^6.$ Similarly,
$(\xi_1,\xi_2,\eta_1,\eta_2)\in A^s_{m,n,p,j}(k,l,a)\Longrightarrow \eta_1,\eta_2\in I_{k,m}\Longrightarrow $ \newline $2^{-k-1}<|K(1,\eta_i)|< 2^{-k+2},$ for $i=1,2.$ Then
 we have the same inequality for $\eta_3, $ since $\eta_1<\eta_3<\eta_2$ and $I_{k,m}$ is an interval. Also, $(\xi_1,\xi_2,\eta_1,\eta_2)\in A^s_{m,n,p,j}(k,l,a)$ \newline $\Longrightarrow \xi_1,\xi_2\in \overset{\:\circ}{\mathcal I}_p\cup \overset{\:\circ}{\tilde{\mathcal I}}_p\Longrightarrow \frac{1}{\lambda}<|\xi_1|,|\xi_2|<\lambda^2.$ Putting all these bounds together we have, for $(\xi_1,\xi_2,\eta_1,\eta_2)\in A^s_{m,n,p,j}(k,l,a),$
\begin{equation}\label{eq*13}
\frac{1}{(d-1)\alpha^6 \lambda^{2d}}2^{-k}|\eta_2-\eta_1|<|JV(\xi_1,\xi_2,\eta_1,\eta_2)|<\frac{8\alpha^6\lambda^{4d}}{d-1} 2^{-k}|\eta_2-\eta_1|.
\end{equation}
It follows immediately that  $|JV|$ is bounded below, with lower bound
$\frac{2^{-k}a}{(d-1)\alpha^6\lambda^{2d}}.$ 

 \medskip

As a preliminary result, we show that $V$ is locally one-to-one.

\begin{lemma}\label{lemma2.3}
For fixed $m,n,p,j, s$ the map $V:A^s_{m,n,p,j}(k,l,a)\rightarrow {\R}^4$ is locally one-to-one on $A^s_{m,n,p,j}(k,l,a),$ i.e, given $(\xi_1,\xi_2,\eta_1,\eta_2)\in A^s_{m,n,p,j}(k,l,a)$ there is some $\epsilon_0>0$ such that $V$ is one-to-one on $B((\xi_1,\xi_2,\eta_1,\eta_2),\epsilon_0).$ (We note that $\epsilon_0$ may depend on $(\xi_1,\xi_2,\eta_1,\eta_2),s, m,n,p,j,k,l$ and $a.$)

\end{lemma}

\begin{proof}
It is easily seen that each $A^s_{m,n,p,j}(k,l,a)$ is an open set, by its construction. Therefore, since $|JV(\xi_1,\xi_2,\eta_1,\eta_2)|$ is positive on
$A^s_{m,n,p,j}(k,l,a),$  the Inverse Function Theorem gives the result. 

\end{proof}

In order to prove that $V$ is many-to-one on the region of integration under consideration we will need a result from algebraic topology and real algebraic geometry. This result will prove useful several times in the course of our proof. 

 Let $X$ be a set in $\R^n,$ defined by polynomial inequalities
$$f_1(x_1,\dots, x_n)\geq 0, \dots, f_p(x_1,\dots, x_n)\geq 0.$$
Then the $q$th {\it Betti number} of $X$ is defined to be the rank of the \v{C}ech cohomology group 
  $H^q (X),$ with coefficients in some fixed field $F.$ 
  Also, the sum of the Betti numbers of $X$ is defined to be
  $\text{rank } H^* X.$ (See \cite{BCR}, \cite{H} for the relevant definitions.) 
  
  We now state the following theorem, due to Milnor.
  
  \begin{theorem}\label{milnor} \cite{M} (Theorem 3)
  If $X\subset \R^n$ is defined by polynomial inequalities of the form
  $$f_1\geq 0,\dots,f_p\geq 0$$
  with total degree $D=\text{deg }f_1+\cdots+\text{deg }f_p,$ then
  $$\text{rank } H^* X\leq \frac12 (2+D)(1+D)^{n-1}.$$
  \end{theorem}
  
  An immediate corollary of this is 
  
  \begin{corollary} \label{alggeom}
  If $X$ is as in Milnor's theorem, then the number of connected components
  of $X$ is $\leq \frac{1}{2}(2+D)(1+D)^{n-1}.$ Moreover, each connected component is path-connected.
  \end{corollary}
  
  \begin{proof}
  Let $b_0$ be the $0^{\text{th}}$ Betti number of $X.$ Then $b_0$ is the number of connected
  components of $X.$ See \cite{BCR}, Theorem 2.4.5, Proposition 2.5.11 and Remark 11.5.5. 
 Another source for these last results is \cite{H}, Proposition 2.7 and Chapter 3, p.198. 
\end{proof}  
 
With $m,n,p,j,s$ still fixed, we now fix $(u_0,v_0,w_0)$ such that there exists $z_0$ with  
$(u_0,v_0,w_0,z_0)\in V(A^s_{m,n,p,j}(k,l,a)).$ We use the usual notation that 
\begin{eqnarray*}
& & V(\xi_1,\xi_2,\eta_1,\eta_2)\\
& & =(V_1(\xi_1,\xi_2,\eta_2,\eta_2),V_2(\xi_2,\xi_2,\eta_1,\eta_2),V_3(\xi_1,\xi_2,\eta_1,\eta_2),V_4(\xi_2,\xi_2,\eta_1,\eta_2)).
\end{eqnarray*}
Then, for simplicity of notation, we let $A$ 
denote $A^s_{m,n,p,j}(k,l,a)$ and we define
\begin{eqnarray*}
A(u_0,v_0,w_0) & = & \{(\xi_2,\xi_2,\eta_1,\eta_2)\in A \mid V_1(\xi_1,\xi_2,\eta_1,\eta_2)=u_0, \\
& & \qquad\qquad \qquad V_2(\xi_1,\xi_2,\eta_1,\eta_2)=v_0, V_3(\xi_1,\xi_2,\eta_1,\eta_2)=w_0\}.
\end{eqnarray*}

\begin{lemma}\label{lemma2.5}
For any $\overline{z}\in V_4(A(u_0,v_0,w_0)),$
$$\#\{(\xi_1,\xi_2,\eta_1,\eta_2)\in A\mid (\xi_1,\xi_2,\eta_1,\eta_2)\in V^{-1}(u_0,v_0,w_0,\overline{z})\}\leq \overline{N}(d),$$
where $\overline{N}(d)$ is a constant depending only on $d.$
\end{lemma}

\begin{proof}
Let $B=\{(\xi_1,\xi_2,\eta_1,\eta_2)\in A\mid
(\xi_1,\xi_2,\eta_1,\eta_2)\in V^{-1}(u_0,v_0,w_0,\overline{z})\}.$
Then $B$ is defined by a finite set of polynomial
inequalities. For example, if $s=1, j\neq 0,$ then, if we let $\overset{\:\circ}{\mathcal I}_p=(a_p,b_p), I_{k,m}=(c_{k,m},d_{k,m}),$ and $I_{l,n}=(e_{l,n},f_{l,n}),$ $B$ is given by the polynomial inequalities
\begin{eqnarray*}
&  \xi_1+\xi_2\geq u_0 & \xi_1+\xi_2\leq u_0\\
&  \xi_1\eta_1+\xi_2\eta_2\geq v_0  & \xi_1\eta_1+\xi_2\eta_2\leq v_0\\
&  \xi_1^d\phi(1,\eta_1)+\xi_2^d\phi(1,\eta_2)\geq w_0 &\xi_1^d\phi(1,\eta_1)+\xi_2^d\phi(1,\eta_2)\leq w_0\\
&  \xi_2^d\phi(1,\eta_2)-\xi_1^d\phi(1,\eta_1)\geq \overline{z} & \xi_2^d\phi(1,\eta_2)-\xi_1^d\phi(1,\eta_1)\leq \overline{z}\\
&  -1<\eta_1<1 & -1<\eta_2<1\\
& \tan (\theta_j-\gamma_0)< \eta_1<\tan(\theta_j+\gamma_0) & \tan (\theta_j-\gamma_0) < \eta_2<\tan(\theta_j+\gamma_0)\\
& a_p<\xi_1<b_p & a_p<\xi_2<b_p\\
& c_{k,m}<\eta_1<d_{k,m} & c_{k,m}<\eta_2<d_{k,m}\\
& e_{l,n}<\eta_1<f_{l,n} & e_{l,n}<\eta_2<f_{l,n}\\
& \eta_2-\eta_1>a.&
\end{eqnarray*}
The other cases are similar. 

By Corollary \ref{alggeom}, if $D$ represents the total degree of the polynomials, then the number of connected components of $B$ is $\leq \frac12(2+D)(1+D)^3.$ Since $D$ depends only on $d,$ if we take $\overline{N}(d)=\frac12 (2+D)(1+D)^3$ 
then the number of connected components of
$B$ is $\leq \overline{N}(d).$ Moreover, by Corollary \ref{alggeom} each connected component of $B$ is path-connected. We let $C$ be a path-connected component of $B$ and take $(\xi_1,\xi_2,\eta_1,\eta_2)\in C.$ Then, by Lemma 
\ref{lemma2.3}, $V$ is locally one-to-one and so there is some $\epsilon_0>0$ such that 
$B((\xi_1,\xi_2,\eta_1,\eta_2),\epsilon_0)$ contains no other points of $C.$ Since $C$ is path-connected we must have $C=\{(\xi_1,\xi_2,\eta_1,\eta_2)\}.$  Thus the number of elements of $B$ is $\leq \overline{N}(d).$ We note that although in Lemma \ref{lemma2.3} the size of the ball depends on many parameters that dependence plays no role here.

\end{proof}

We now come to the crux of the proof. We recall that, for $j\in J_1,$ \eqref{lowerbound} holds. It follows that  whenever
$(\xi_1,\xi_2,\eta_1,\eta_2)\in A^s_{m,n,p,j}(k,l,a)$ we have
$\frac{|K(1,\eta_i)|}{|\phi'(1,\eta_i)^2|}\geq \epsilon,$ for $i=1,2.$

 \begin{proposition}\label{prop2.5}
 We let $A$ represent one of the sets $A^s_{m,n,p,j}(k,l,a),$ with $j\in J_1.$ Then
  \begin{eqnarray*}
 & & \int_{V_4(A(u_0,v_0,w_0))}\sum_{V^{-1}(u_0,v_0,w_0,z)}\frac{|K(\xi_1,\xi_1\eta_1)K(\xi_2,\xi_2\eta_2)|^{\frac14}}{|JV(\xi_1,\xi_2,\eta_1,\eta_2)|}\Psi_{0lk}(\xi_1,\xi_1\eta_1)^2\cdot \\
 & & \qquad\qquad\qquad \qquad  \Psi_{0lk}(\xi_2,\xi_2\eta_2)^2 |\xi_1\xi_2| \chi_A(\xi_1,\xi_2,\eta_1,\eta_2) dz \leq C,
 \end{eqnarray*}
  where $C$ is independent of
 $ k,l,a, u_0,v_0,w_0,$ depending only on $d $ and $\epsilon.$ (We use $\sum_{V^{-1}(u_o,v_0,w_0,z)}$ to denote $\sum_{(\xi_1,\xi_2,\eta_1,\eta_2)\in V^{-1}(u_0,v_0,w_0,z)}.$)
 
 \end{proposition}
 
 Before proving Proposition \ref{prop2.5} we have the following corollary.
  
 \begin{corollary}\label{corollary2.6}
 a) Let $j\in J_1.$ Then 
 $$\|U(t)S_0^{\lambda} Q^\alpha_l P_k W_jg\|_{L^4_{xyt}}\leq C\|g\|_{L^2_{xy}},\ \forall g\in L^2_{xy},$$
 with $C$ independent of $k$ and $l,$ depending only on $d$ and $\epsilon.$ 
 
 b) Let $j\in J_1.$ Then
 $$\|U(t)W_j g\|_{L^4_{xyt}}\leq C\|g\|_{L^2_{xy}}, \forall g\in L^2(\R^2), 
 $$
 with $C$ independent of $k$ and $l,$ depending only on $d$ and $\epsilon.$
 \end{corollary}
 
 \begin{proof}
 We first note that b) is immediate from a), by Proposition
 \ref{prop1}.
 
  It suffices to assume $\hat{g}\in C_0^{\infty} ({\R}^2),$ by density. From \cite{Simon}, p.47, we have the following change of variables formula:
  \begin{equation}\label{*27}
 \int_{{\R}^4}Jf(x)G(x)dx = \int_{{\R}^4}\left(\int_{f^{-1}(y)}G d\mathcal{H}^{0}\right) dy
 \end{equation}
 where $\mathcal{H}^{0}$ denotes the $0-$dimensional Hausdorff measure (counting measure), $G$ is non-negative and Borel-measurable, $f\in C^1,$ and $Jf$ denotes the Jacobian of $f.$ 
 Thus, if $\#\{f^{-1}(y)\}=N_f(y)<\infty,$ then we have
  \begin{equation}\label{*28}
 \int_{{\R}^4}Jf(x)G(x) dx=\int_{{\R}^4}\left(\sum_{x\in f^{-1}(y)}G (x)\right) dy.
 \end{equation}
 
 We first show that \eqref{*27} and \eqref{*28} hold under the
 alternative assumptions that $G$ is bounded, Borel-measurable,
 supported on a compact set $K,$ $f$ is such that \newline
 $\sup_y\#(\{f^{-1}(y)\}\cap K) =N_f(K)<\infty,$
  and $f\in C^1(K).$  To show this we take such a $G$ and write it as 
 $G=G^+ + G^-$ where $G^+ (x)=\frac{|G(x)|+G(x)}{2}\geq 0$ and $G^- (x)=\frac{|G(x)|-G(x)}{2}\geq 0.$
 Then, since $G^+, G^-$ are non-negative and $G$ is  Borel measurable, we have, by \eqref{*27},
 $$\int_{{\R}^4}Jf(x)G^{\pm} (x)dx =
 \int_{{\R}^4}\left(\int_{f^{-1}(y)} G^{\pm}d \mathcal{H}^{0}\right) dy,$$
 and both quantities are finite since $Jf$ is bounded on $K$ and $G^{\pm}$ are bounded and supported on $K.$ 
 Now $Jf$ is bounded on $K$ (since $f\in C^1(K)$), $N_f(K)<\infty,$ and $G^{\pm}$ are bounded and supported on the compact set $K$, and hence
 \begin{eqnarray*}
 \int_{{\R}^4}Jf(x)G(x)dx & = & \int_{{\R}^4}Jf(x)G^+(x)dx -\int_{{\R}^4}Jf(x)G^{-}(x)dx\\
 & = & \int_{{\R}^4}\left(\sum_{x\in f^{-1}(y)}G^{+}(x)\right)dy-\int_{{\R}^4}\left(\sum_{x\in f^{-1}(y)}G^{-}(x)\right)dy\\
 & = & \int_{{\R}^4}\left(\sum_{x\in f^{-1}(y)}G(x)\right) dy\\
 & = & \int_{{\R}^4}\left(\int_{f^{-1}(y)} G d \mathcal{H}^{0}\right)dy.
 \end{eqnarray*}
 
 Now we recall the definition of $V$ as given in \eqref{eq*10} and \eqref{eq*11}. We let $A$ be as in Proposition \ref{prop2.5}. $V$ is clearly $C^1$ on $A,$ and hence on any compact subset of $A.$ 
 
  We now define
 \begin{eqnarray*}
 & & h(\xi_1,\xi_2,\eta_1,\eta_2)\\
 & = &  e^{it(\xi_1^d\phi(1,\eta_1)+\xi_2^d\phi(1,\eta_2))} e^{i[x(\xi_1+\xi_2)+y(\xi_1\eta_1+\xi_2\eta_2)]}\frac{|\xi_1\xi_2|^{\frac{d-2}{4}}|K(1,\eta_1)K(1,\eta_2)|^{\frac18}}{JV(\xi_1,\xi_2,\eta_1,\eta_2)}\cdot \\
 & & \Psi_{0lk}(\xi_1,\xi_1\eta_1)\Psi_{0lk}(\xi_2,\xi_2\eta_2)\hat{g}(\xi_1,\xi_1\eta_1)\hat{g}(\xi_2,\xi_2\eta_2)\xi_1\xi_2.
 \end{eqnarray*}
 Then, we notice that on $A,$ using also the fact that $\hat{g}\in C_0^{\infty},$ 
 $$|h(\xi_1,\xi_2,\eta_1,\eta_2)|\leq C \frac{2^{\frac{3k}{4}}}{a},$$
 where $C$ depends only on $d, \epsilon$ and the bounds of $\hat{g}.$ Thus $h$ is bounded on $A$, continuous, compactly supported and, by Lemma \ref{lemma2.5}, we have also $\#\{V^{-1}(u,v,w,z)\in A\}=N_V (u,v,w,z) \leq \overline{N}(d).$ Then, by \eqref{*28},
  we have (first replacing $A$ with $A_\alpha$ for $A_\alpha $ compact, nested ($A_{\alpha+1}\supseteq A_{\alpha})$ and satisfying $A=\bigcup_{\alpha=1}^{\infty} A_{\alpha}$, and then passing to the limit)
  \begin{eqnarray*}
& &  \int_A e^{it(\xi_1^d\phi(1,\eta_1)+\xi_2^d\phi(1,\eta_2))} e^{i[x(\xi_1+\xi_2)+y(\xi_1\eta_1+\xi_2\eta_2)]}|\xi_1\xi_2|^{\frac{d-2}{4}}|K(1,\eta_1)K(1,\eta_2)|^{\frac18}\cdot \\
&  & \quad \Psi_{0lk}(\xi_1,\xi_1\eta_1)\Psi_{0lk}(\xi_2,\xi_2\eta_2)\hat{g}(\xi_1,\xi_1\eta_1)\hat{g}(\xi_2,\xi_2\eta_2)\xi_1\xi_2 d\xi_1d\xi_2 d\eta_1 d\eta_2\\
& = &\int JV(\xi_1,\xi_2,\eta_1,\eta_2)\chi_A(\xi_1,\xi_2,\eta_1,\eta_2) h(\xi_1,\xi_2,\eta_1,\eta_2)d\xi_1 d\xi_2d\eta_1 d\eta_2\\
& =  & \int\int_{V^{-1}(u,v,w,z)}\chi_A h d\mathcal{H}^{\circ} du dv dw dz\\
& = & \int e^{i(xu+yv+tw)} 
\bigg(\sum_{(\xi_1,\xi_2,\eta_1,\eta_2)\in V^{-1}(u,v,w,z)}\frac{|\xi_1\xi_2|^{\frac{d-2}{4}}|K(1,\eta_1)K(1,\eta_2)|^{\frac18}}{JV(\xi_1,\xi_2,\eta_1,\eta_2)}\cdot \\
& & \quad \Psi_{0lk}(\xi_1,\xi_1\eta_1) \Psi_{0lk}(\xi_2,\xi_2\eta_2)\hat{g}(\xi_1,\xi_1\eta_1)\hat{g}(\xi_2,\xi_2\eta_2) \chi_A(\xi_1,\xi_2,\eta_1,\eta_2)\xi_1\xi_2 \bigg)\\
& & \qquad dudvdwdz\\
& =  & \int e^{i(xu+yv+tw)} L(u,v,w) dudvdw,
\end{eqnarray*}
where 
\begin{eqnarray*}
& & L(u,v,w)=\\
& & \int \bigg(\sum_{(\xi_1,\xi_2,\eta_1,\eta_2)\in V^{-1}(u,v,w,z)}\frac{|\xi_1\xi_2|^{\frac{d-2}{4}}|K(1,\eta_1)K(1,\eta_2)|^{\frac18}}{JV(\xi_1,\xi_2,\eta_1,\eta_2)} 
\Psi_{0lk}(\xi_1,\xi_1\eta_1)\cdot \\
& &  \qquad \Psi_{0lk}(\xi_2,\xi_2\eta_2)\hat{g}(\xi_1,\xi_1\eta_1)\hat{g}(\xi_2,\xi_2\eta_2)\chi_A (\xi_1,\xi_2,\eta_1,\eta_2)\xi_1\xi_2 \bigg)dz
\end{eqnarray*}
Thus, recalling \eqref{eq*3} in Lemma \ref{lemma2.2}, it suffices to show
$$\left\|\int e^{i(xu+yv+yw)}L(u,v,w)dudvdw\right\|_{L^2_{xyt}}\leq C\|g\|^2_{L^2_{xy}},$$
with $C$ independent of $k,l$ and $a.$ By Plancherel's theorem, this is equivalent to
\begin{eqnarray*}
& & \bigg \|\int\bigg(\sum_{V^{-1}(u,v,w,z)}\frac{|\xi_1\xi_2|^{\frac{d-2}{4}}|K(1,\eta_1)K(1,\eta_2)|^{\frac18}}{JV(\xi_1,\xi_2,\eta_1,\eta_2)} \Psi_{0lk}(\xi_1,\xi_1\eta_1)\cdot \\
& & \quad\Psi_{0lk}(\xi_2,\xi_2\eta_2)\hat{g}(\xi_1,\xi_1\eta_1)\hat{g}(\xi_2,\xi_2\eta_2)\chi_A(\xi_1,\xi_2,\eta_1,\eta_2)\xi_1\xi_2  \bigg)dz\bigg\|_{L^2_{uvw}}\\
& & \leq C\|g\|^2_{L^2_{xy}}.
\end{eqnarray*}
Now, using Cauchy-Schwarz in the sum, followed by Cauchy-Schwarz in the $z$-\newline integral
gives us
\begin{eqnarray*}
& & \bigg|\int \bigg(\sum_{V^{-1}(u,v,w,z)}\frac{|\xi_1\xi_2|^{\frac{d-2}{4}}|K(1,\eta_1)K(1,\eta_2)|^{\frac18}}{JV(\xi_1,\xi_2,\eta_1,\eta_2)} \Psi_{0lk}(\xi_1,\xi_1\eta_1)\cdot \\
& & \quad \Psi_{0lk}(\xi_2,\xi_2\eta_2)\hat{g}(\xi_1,\xi_1\eta_1)\hat{g}(\xi_2,\xi_2\eta_2)\chi_A(\xi_1,\xi_2,\eta_1,\eta_2)\xi_1\xi_2 \bigg)dz\bigg|^2 \\
& & \leq \Bigg(\int\left(\sum_{V^{-1}(u,v,w,z)}\frac{|\hat{g}(\xi_1,\xi_1\eta_1)\hat{g}(\xi_2,\xi_2\eta_2)|^2}{|JV(\xi_1,\xi_2,\eta_1,\eta_2)|}\chi_A(\xi_1,\xi_2,\eta_1,\eta_2) |\xi_1\xi_2|\right)^{\frac12}\cdot\\
& & \bigg(\sum_{V^{-1}(u,v,w,z)}\frac{|K(1,\eta_1)K(1,\eta_2)|^{\frac14}|\xi_1\xi_2|^{\frac{d-2}{2}}}{|JV(\xi_1,\xi_2,\eta_1,\eta_2)|}\Psi_{0lk}(\xi_1,\xi_1\eta_1)^2\Psi_{0lk}(\xi_2,\xi_2\eta_2)^2\cdot \\
& & \qquad  \chi_A (\xi_1,\xi_2,\eta_1,\eta_2) |\xi_1\xi_2|\bigg)^{\frac12}dz\Bigg)^2\\
& & \leq \left(\int\sum_{V^{-1}(u,v,w,z)}\frac{|\hat{g}(\xi_1,\xi_1\eta_1)\hat{g}(\xi_2,\xi_2\eta_2)|^2}{|JV(\xi_1,\xi_2,\eta_1,\eta_2)|}\chi_A(\xi_1,\xi_2,\eta_1,\eta_2) |\xi_1\xi_2|dz\right)\cdot\\
& & \Bigg(\int \sum_{V^{-1}(u,v,w,z)}\frac{|K(1,\eta_1)K(1,\eta_2)|^{\frac14}|\xi_1\xi_2|^{\frac{d-2}{2}}}{|JV(\xi_1,\xi_2,\eta_1,\eta_2)|}\Psi_{0lk}(\xi_1,\xi_1\eta_1)^2\Psi_{0lk}(\xi_2,\xi_2\eta_2)^2\cdot\\
& & \qquad\chi_A(\xi_1,\xi_2,\eta_1,\eta_2) |\xi_1\xi_2|  dz\Bigg).
\end{eqnarray*}

Then
\begin{eqnarray}
& & \|L\|^2_{L^2_{uvw}}\label{eq*50}\\
& & \leq \left(\int\sum_{V^{-1}(u,v,w,z)}\frac{|\hat{g}(\xi_1,\xi_1\eta_1)\hat{g}(\xi_2,\xi_2\eta_2)|^2}{|JV(\xi_1,\xi_2,\eta_1,\eta_2)|}\chi_A(\xi_1,\xi_2,\eta_1,\eta_2) |\xi_1\xi_2|dudvdwdz\right)\cdot\nonumber\\
& & \sup_{u,v,w}\bigg(\int\chi_{V(A)}(u,v,w,z)
\sum_{V^{-1}(u,v,w,z)}\frac{|K(1,\eta_1)K(1,\eta_2)|^{\frac14}|\xi_1\xi_2|^{\frac{d-2}{2}}}{|JV(\xi_1,\xi_2,\eta_1,\eta_2)|}\cdot \nonumber\\
& & \qquad\qquad \Psi_{0lk}(\xi_1,\xi_1\eta_1)^2\Psi_{0lk}(\xi_2,\xi_2\eta_2)^2\chi_A(\xi_1,\xi_2,\eta_1,\eta_2)|\xi_1\xi_2| dz\bigg).\nonumber
\end{eqnarray}

By Proposition \ref{prop2.5} we have, for fixed $(u_0,v_0,w_0),$ 
 \begin{eqnarray*}
 &&  \int_{V_4((A(u_0,v_0,w_0))}\sum_{V^{-1}(u_0,v_0,w_0,z)}\frac{|K(\xi_1,\xi_1\eta_1)K(\xi_2,\xi_2\eta_2)|^{\frac14}}{|JV(\xi_1,\xi_2,\eta_1,\eta_2)|}\Psi_{0lk}(\xi,\xi_1\eta_1)^2\cdot  \\
 & & \qquad\qquad \qquad\qquad \Psi_{0lk}(\xi_2,\xi_2\eta_2)^2 \chi_A(\xi_1,\xi_2,\eta_1,\eta_2) |\xi_1\xi_2| dz \leq C,
 \end{eqnarray*}
 with $C$ independent of
 $a,k,l,u_0,v_0,w_0,$ depending only on $d$ and $\epsilon.$
 
 It follows that 
 \begin{eqnarray*}
 & & \sup_{u,v,w}\bigg(\int\chi_{V(A)}(u,v,w,z)
\sum_{V^{-1}(u,v,w,z)}\frac{|K(1,\eta_1)K(1,\eta_2)|^{\frac14}|\xi_1\xi_2|^{\frac{d-2}{2}}}{|JV(\xi_1,\xi_2,\eta_1,\eta_2)|}\cdot \\
& & \qquad\qquad \Psi_{0lk}(\xi_1,\xi_1\eta_1)^2\Psi_{0lk}(\xi_2,\xi_2\eta_2)^2\chi_A(\xi_1,\xi_2,\eta_1,\eta_2)|\xi_1\xi_2|  dz\bigg)\\
& & \leq C,
\end{eqnarray*}
with $C$ independent of $a,k$ and $l.$ If we use this in \eqref{eq*50}, together with the fact that 
\begin{eqnarray*}
& & \int\sum_{V^{-1}(u,v,w,z)}\frac{|\hat{g}(\xi_1,\xi_1\eta_1)\hat{g}(\xi_2,\xi_2\eta_2)|^2}{|JV(\xi_1,\xi_2,\eta_1,\eta_2)|}\chi_A(\xi_1,\xi_2,\eta_1,\eta_2)|\xi_1\xi_2|dudvdwdz\\
& =& \int |\hat{g}(\xi_1,\xi_1\eta_1)\hat{g}(\xi_2,\xi_2\eta_2)|^2 \chi_A(\xi_1,\xi_2,\eta_1,\eta_2) |\xi_1\xi_2|d\xi_1 d\xi_2 d\eta_1 d\eta_2\\
&\leq  & \|g\|_{L^2_{xy}}^4,
\end{eqnarray*}
 we obtain
$$\|L\|_{L^2_{uvw}}\leq C\|g\|_{L^2_{xy}}^2,$$
with $C$ independent of $a,k$ and $l,$ as we needed.

 \end{proof}

We now turn to the proof of Proposition \ref{prop2.5}. 

\subsection{Proof of Proposition \ref{prop2.5}}

By a similar argument to that given in the proof of Lemma
\ref{lemma2.5}, each $A(u_0,v_0,w_0)$ is defined by a finite set of polynomial inequalities and so has at most $\tilde{N}(d)$ path-connected components. We write 
$$A(u_0,v_0,w_0)=\bigcup_{r=1}^{\tilde{N}(d)} A_r (u_0,v_0,w_0),$$
where each $A_r (u_0,v_0,w_0)$ is path-connected. Then we also have
$$V_4(A(u_0,v_0,w_0))=\bigcup_{r=1}^{\tilde{N}(d)}V_4(A_r(u_0,v_0,w_0)).$$
Since $V_4$ is continuous and each $A_r(u_0,v_0,w_0)$ is connected, it follows that each $V_4(A_r(u_0,v_0,w_0))$ is connected and hence is an interval. We note that these intervals need not be disjoint.
Then
\begin{eqnarray*}
 & & \int_{V_4((A(u_0,v_0,w_0))}\sum_{V^{-1}(u_0,v_0,w_0,z)}\frac{|K(\xi_1,\xi_1\eta_1)K(\xi_2,\xi_2\eta_2)|^{\frac14}}{|JV(\xi_1,\xi_2,\eta_1,\eta_2)|}\Psi_{0lk}(\xi_1,\xi_1\eta_1)^2\cdot  \\
 & &\qquad\qquad\qquad \Psi_{0lk}(\xi_2,\xi_2\eta_2)^2  \chi_A(\xi_1,\xi_2,\eta_1,\eta_2)|\xi_1\xi_2|  dz \leq \\
 &  & \sum_{r=1}^{\tilde{N}(d)}\int_{V_4(A_r (u_0,v_0,w_0))}\sum_{V^{-1}(u_0,v_0,w_0,z)}\frac{|K(\xi_1,\xi_1\eta_1)K(\xi_2,\xi_2\eta_2)|^{\frac14}}{|JV(\xi_1,\xi_2,\eta_1,\eta_2)|}\Psi_{0lk}(\xi_1,\xi_1\eta_1)^2\cdot  \\
& & \qquad\qquad\qquad \Psi_{0lk}(\xi_2,\xi_2\eta_2)^2  \chi_A(\xi_1,\xi_2,\eta_1,\eta_2)|\xi_1\xi_2| dz,
 \end{eqnarray*}
 and so it suffices to show that
 \begin{eqnarray}
 & & \int_{V_4(A_r (u_0,v_0,w_0))}\sum_{V^{-1}(u_0,v_0,w_0,z)}\frac{|K(\xi_1,\xi_1\eta_1)K(\xi_2,\xi_2\eta_2)|^{\frac14}}{|JV(\xi_1,\xi_2,\eta_1,\eta_2)|}\Psi_{0lk}(\xi_1,\xi_1\eta_1)^2\cdot \label{eq*20}\\
& & \qquad\qquad\Psi_{0lk}(\xi_2,\xi_2\eta_2)^2\chi_A(\xi_1,\xi_2,\eta_1,\eta_2) |\xi_1\xi_2|dz\leq C,\nonumber
\end{eqnarray}
with $C$ independent of $a,k, l, u_0,v_0$ and $w_0.$

In order to prove \eqref{eq*20} we first show that $V_4$ is one-to-one on each $A_r(u_0,v_0,w_0).$ This requires a couple of lemmas.

\begin{lemma}\label{lemma2.7}
Each $A_r (u_0,v_0,w_0)$ is a connected one-dimensional $C^\infty$ manifold.
\end{lemma}
 
\begin{proof}
It is routine to show that each $A(u_0,v_0,w_0)$ is a one-dimensional $C^\infty$ manifold, using the fact that $V$ is locally one-to-one (Lemma \ref{lemma2.3}). 
Let $\overline{x}\in A(u_0,v_0,w_0).$ Then, by the Inverse Function Theorem, there is an open set $U_{\overline{x}}$ such that $V:U_{\overline{x}}\rightarrow V(U_{\overline{x}})$ is a diffeomorphism. Then we can find an open cube $Q_{V(\overline{x})},$ with centre $V(\overline{x})=(u_0,v_0,w_0,V_4 (\overline{x})),$ such that $Q_{V(\overline{x})}\subseteq V(U_{\overline{x}}).$  Then 
$V^{-1}(Q_{V(\overline{x})})\cap A(u_0,v_0,w_0)$ is open in $A(u_0,v_0,w_0).$ We define
$\phi_{\overline{x}}: V^{-1}(Q_{V(\overline{x})})\cap A(u_0,v_0,w_0)\rightarrow {\mathbb R}$ by $\phi_{\overline{x}}(x)=V_4(x).$ Then it is easily checked that \newline
$\{V^{-1}(Q_{V(\overline{x})})\cap A(u_0,v_0,w_0),\phi_{\overline{x}}\}_{\overline{x}\in A(u_0,v_0,w_0)}$ is a chart for $A(u_0,v_0,w_0).$ Furthermore,  $\{V^{-1}(Q_{V(\overline{x})})\cap A_r(u_0,v_0,w_0), \phi_{\overline{x}}\}_{\overline{x}\in A_r(u_0,v_0,w_0)}$ gives a chart for \newline $A_r(u_0,v_0,w_0),$ and so each connected component, $A_r(u_0,v_0,w_0),$ of $A(u_0,v_0,w_0)$ is also a one-dimensional connected $C^{\infty} $ manifold.

\end{proof}

\begin{lemma}\label{lemma2.8}
Given distinct points $(\xi_1^0,\xi_2^0,\eta_1^0,\eta_2^0), (\overline{\xi}_1,\overline{\xi}_2,\overline{\eta}_1,\overline{\eta}_2)\in A_r(u_0,v_0,w_0),$ there is a continuous one-to-one map
$\varphi:[0,1]\rightarrow A_r(u_0,v_0,w_0)$ such that  $\varphi(0)=(\xi_1^0,\xi_2^0,\eta_1^0,\eta_2^0)$ and $\varphi(1)= (\overline{\xi}_1,\overline{\xi}_2,\overline{\eta}_1,\overline{\eta}_2).$
\end{lemma}

\begin{proof}
Since $A_r(u_0,v_0,w_0)$ is a connected one-dimensional $C^\infty$ manifold (by Lemma \ref{lemma2.7}),
it is homeomorphic (in fact diffeomorphic) to either an interval or the unit circle $S^1.$ (See \cite{milnor2}, Appendix p.55).

If $A_r(u_0,v_0,w_0)$ is homeomorphic to an interval then we let $I$ be an interval and 
$f:A_r(u_0,v_0,w_0)\rightarrow I$ be a continuous bijection with continuous inverse.
Then we define $\psi:[0,1]\rightarrow I$ by 
$\psi(t)=(1-t)f(\xi_1^0,\xi_2^0,\eta_1^0,\eta_2^0)+tf(\overline{\xi}_1,\overline{\xi}_2,\overline{\eta}_1,\overline{\eta}_2).$ We let $\varphi(t)=f^{-1}(\psi(t)).$ 

If $A_r(u_0,v_0,w_0)$ is homeomorphic to $S^1$ then we let $g:A_r(u_0,v_0,w_0)\rightarrow S^1$ be a continuous bijection with continuous inverse.
Then we define $\tilde{\psi}:[0,1]\rightarrow S^1$ by $\tilde{\psi}(t)=g(\xi_1^0,\xi_2^0,\eta_1^0,\eta_2^0)^{1-t}g(\overline{\xi}_1,\overline{\xi}_2,\overline{\eta}_1,\overline{\eta}_2)^{t}.$   We let 
$\varphi(t)=g^{-1}(\tilde{\psi}(t)).$ 

In both cases, it is easily checked that 
$\varphi$ is a continuous one-to-one map into $A_r(u_0,v_0,w_0),$  with  $\varphi(0)=(\xi_1^0,\xi_2^0,\eta_1^0,\eta_2^0)$ and $\varphi(1)= (\overline{\xi}_1,\overline{\xi}_2,\overline{\eta}_1,\overline{\eta}_2).$
\end{proof}

\begin{proposition}\label{prop2.9}
$V_4$ is one-to-one on each $A_r(u_0,v_0,w_0).$ 
\end{proposition}

\begin{proof}
By Lemma \ref{lemma2.3}, $V$ is locally one-to-one on $A.$ Thus, if  $(\xi_1,\xi_2,\eta_1,\eta_2)$  $ \in A_r(u_0,v_0,w_0),$   then there is some $\epsilon_0>0$ such that $V$ is one-to-one on \newline $B((\xi_1,\xi_2,\eta_1,\eta_2),\epsilon_0)\cap A.$ Moreover,
$V_1(\xi_1,\xi_2,\eta_1,\eta_2)  =  u_0;
V_2(\xi_1,\xi_2,\eta_1,\eta_2)  =  v_0;$\newline $
V_3(\xi_1,\xi_2,\eta_1,\eta_2)  =  w_0.$
If $V_4(\xi'_1,\xi'_2,\eta'_1,\eta'_2)=V_4(\xi''_1,\xi''_2,\eta''_1,\eta''_2)$ for 
$(\xi'_1,\xi'_2,\eta'_1,\eta'_2),$ \newline $(\xi''_1,\xi''_2,\eta''_1,\eta''_2) \in $ $ B((\xi_1,\xi_2,\eta_1,\eta_2),\epsilon_0)\cap A(u_0,v_0,w_0),$ then \newline $V(\xi'_1,\xi'_2,\eta'_1,\eta'_2)=V(\xi''_1,\xi''_2,\eta''_1,\eta''_2)$  and hence $(\xi'_1,\xi'_2,\eta'_1,\eta'_2)=(\xi''_1,\xi''_2,\eta''_1,\eta''_2).$ Thus $V_4$ is locally one-to-one on  $A_r(u_0,v_0,w_0).$ 
We note that $A_r(u_0,v_0,w_0)$ is connected and $V_4$ is continuous, so $V_4(A_r(u_0,v_0,w_0))$ is an interval, call it ${\mathcal J}_r.$ Note that ${\mathcal J}_r$ has non-empty interior since $V_4$ is locally one-to-one on $A_r(u_0,v_0,w_0)$ and $A_r(u_0,v_0,w_0)$ is a one-dimensional manifold. Suppose, for contradiction,  that $V_4$ is not one-to-one on $A_r(u_0,v_0,w_0).$ Since ${\mathcal J}_r$ has non-empty interior, $A_r(u_0,v_0,w_0)$ must contain more than one point and so we take distinct points
$(\xi'_1,\xi'_2,\eta'_1,\eta'_2)$ and $  (\xi''_1,\xi''_2,\eta''_1,\eta''_2)$  in $A_r(u_0,v_0,w_0)
$ with $V_4(\xi'_1,\xi'_2,\eta'_1,\eta'_2)=V_4(\xi''_1,\xi''_2,\eta''_1,\eta''_2).$ Then, by Lemma \ref{lemma2.8} there is a continuous one-to-one map $\varphi:[0,1]\rightarrow A_r(u_0,v_0,w_0)$ such that $\varphi(0)=(\xi'_1,\xi'_2,\eta'_1,\eta'_2)$ and $\varphi(1)=(\xi''_1,\xi''_2,\eta''_1,\eta''_2).$ 
We define now the map $\nu:[0,1]\rightarrow V_4(A_r(u_0,v_0,w_0))$ by $\nu(t)=V_4(\varphi(t)).$ 
Then, since $V_4$ is locally one-to-one and continuous  and $\varphi$ is one-to-one and continuous, it is readily seen that
$\nu$ is locally one-to-one and continuous. Moreoever, any real-valued continuous, locally one-to-one map defined on an interval is necessarily globally one-to-one. Thus $\nu$ is one-to-one and so  $\nu(1)\neq \nu(0).$ This is a contradiction since $\nu(0)=V_4(\xi'_1,\xi'_2,\eta'_1,\eta'_2)$ and $\nu(1)=V_4(\xi''_1,\xi''_2,\eta''_1,\eta''_2).$ 

\end{proof}

We now return to \eqref{eq*20}. We note that there are at most $\tilde{N}(d)$ $\tilde{r}'$s such that 
$V_4(A_r (u_0,v_0,w_0))\cap V_4(A_{\tilde{r}}(u_0,v_0,w_0))\neq \emptyset$ and, since 
$V_4$ is one-to-one on \newline $A_{\tilde{r}}(u_0,v_0,w_0),$ given $z\in V_4(A_{\tilde{r}}(u_0,v_0,w_0)),$ there is at most one $(\xi_1^{\tilde{r}},\xi_2^{\tilde{r}},\eta_1^{\tilde{r}},\eta_2^{\tilde{r}})$  $\in A_{\tilde{r}}(u_0,v_0,w_0)$ such that 
$V_4(\xi_1^{\tilde{r}},\xi_2^{\tilde{r}},\eta_1^{\tilde{r}},\eta_2^{\tilde{r}})=z.$ Moreover, this
point exists for $\tilde{r}\neq r$ only if $z\in
V_4(A_{\tilde{r}}(u_0,v_0,w_0))\cap V_4(A_r (u_0,v_0,w_0)).$ Then,
letting \newline $I_r=V_4(A_r(u_0,v_0,w_0)),$ we have
\begin{eqnarray*}
&&  \int_{I_r}\sum_{V^{-1}(u_0,v_0,w_0,z)}
  \frac{|K(\xi_1,\xi_1\eta_1)K(\xi_2,\xi_2\eta_2)|^{\frac14}}{|JV(\xi_1,\xi_2,\eta_1,\eta_2)|} \Psi_{0lk}(\xi_1,\xi_1\eta_1)^2
\cdot\\
& &\qquad\qquad\qquad\qquad \Psi_{0lk}(\xi_2,\xi_2\eta_2)^2\chi_A(\xi_1,\xi_2,\eta_1,\eta_2) |\xi_1\xi_2|dz\\
& & = \int_{I_r}\sum_{V^{-1}(u_0,v_0,w_0,z)\cap A}  \frac{|K(\xi_1,\xi_1\eta_1)K(\xi_2,\xi_2\eta_2)|^{\frac14}}{|JV(\xi_1,\xi_2,\eta_1,\eta_2)|} \cdot \\
& & \qquad\qquad\qquad\qquad\qquad \Psi_{0lk}(\xi_1,\xi_1\eta_1)^2\Psi_{0lk}(\xi_2,\xi_2\eta_2)^2|\xi_1\xi_2| dz\\
& & =\int_{I_r}\bigg( \sum_{V^{-1}(u_0,v_0,w_0,z)\cap \left( \bigcup_{\tilde{r}}A_{\tilde{r}}(u_0,v_0,w_0)\right) }
 \frac{|K(\xi_1,\xi_1\eta_1)K(\xi_2,\xi_2\eta_2)|^{\frac14}}{|JV(\xi_1,\xi_2,\eta_1,\eta_2)|}\cdot \\
 & & \qquad\qquad \qquad  \qquad \qquad \Psi_{0lk}(\xi_1,\xi_1\eta_1)^2\Psi_{0lk}(\xi_2,\xi_2\eta_2)^2|\xi_1\xi_2|\bigg)dz\\
& & \leq \int_{I_r}\bigg(  \sum_{\tilde{r}}\sum_{V^{-1}(u_0,v_0,w_0,z)\cap A_{\tilde{r}}(u_0,v_0,w_0)}
\frac{|K(\xi_1,\xi_1\eta_1)K(\xi_2,\xi_2\eta_2)|^{\frac14}}{|JV(\xi_1,\xi_2,\eta_1,\eta_2)|} \cdot \\
& & \qquad\qquad\qquad\qquad\qquad \Psi_{0lk}(\xi_1,\xi_1\eta_1)^2\Psi_{0lk}(\xi_2,\xi_2\eta_2)^2 |\xi_1\xi_2|\bigg)  dz\\
& & = \sum_{\tilde{r}}\bigg( \int_{I_r}\sum_{V^{-1}(u_0,v_0,w_0,z)\cap
  A_{\tilde{r}}(u_0,v_0,w_0)}\frac{|K(\xi_1,\xi_1\eta_1)K(\xi_2,\xi_2\eta_2)|^{\frac14}}{|JV(\xi_1,\xi_2,\eta_1,\eta_2)|}\cdot
  \\
& & \qquad\qquad\qquad\qquad\qquad  \Psi_{0lk}(\xi_1,\xi_1\eta_1)^2\Psi_{0lk}(\xi_2,\xi_2\eta_2)^2|\xi_1\xi_2| dz\bigg) 
\end{eqnarray*}
\begin{eqnarray*}
& & = \sum_{\tilde{r}}\bigg( \int_{I_r\cap I_{\tilde{r}}}\sum_{V^{-1}(u_0,v_0,w_0,z)\cap A_{\tilde{r}}(u_0,v_0,w_0)}\frac{|K(\xi_1,\xi_1\eta_1)K(\xi_2,\xi_2\eta_2)|^{\frac14}}{|JV(\xi_1,\xi_2,\eta_1,\eta_2)|} \cdot \\
& & \qquad\qquad\qquad\qquad\qquad \Psi_{0lk}(\xi_1,\xi_1\eta_1)^2\Psi_{0lk}(\xi_2,\xi_2\eta_2)^2
 |\xi_1\xi_2| dz\bigg)\\
 & & \leq  \sum_{\tilde{r}}\bigg( \int_{I_{\tilde{r}} }\bigg(
  \sum_{V^{-1}(u_0,v_0,w_0,z)\cap A_{\tilde{r}}(u_0,v_0,w_0)}
  \frac{|K(\xi_1,\xi_1\eta_1)K(\xi_2,\xi_2\eta_2)|^{\frac14}}{|JV(\xi_1,\xi_2,\eta_1,\eta_2)|}\cdot \\
& & \qquad\qquad\qquad\qquad\qquad  \Psi_{0lk}(\xi_1,\xi_1\eta_1)^2\Psi_{0lk}(\xi_2,\xi_2\eta_2)^2|\xi_1\xi_2|\bigg)
 dz\bigg)\\
 & & =\sum_{\tilde{r}}\int_{I_{\tilde{r}}}
\frac{|K(\xi^{\tilde{r}}_1,\xi^{\tilde{r}}_1\eta^{\tilde{r}}_1)K(\xi^{\tilde{r}}_2,\xi^{\tilde{r}}_2\eta^{\tilde{r}}_2)|^{\frac14}}{|JV(\xi^{\tilde{r}}_1,\xi^{\tilde{r}}_2,\eta^{\tilde{r}}_1, \eta^{\tilde{r}}_2)|}
 \Psi_{0lk}(\xi^{\tilde{r}}_1,\xi^{\tilde{r}}_1\eta^{\tilde{r}}_1)^2\Psi_{0lk}(\xi^{\tilde{r}}_2,\xi^{\tilde{r}}_2\eta^{\tilde{r}}_2)^2|\xi^{\tilde{r}}_1\xi^{\tilde{r}}_2| dz\\
\end{eqnarray*}
and for each $\tilde{r},$  $V_4:A_{\tilde{r}}(u_0,v_0,w_0)\rightarrow V_4(A_{\tilde{r}}(u_0,v_0,w_0))$ is bijective.
It therefore suffices to show, for all $r,$  
\begin{equation}\label{eq*30}
\int_{V_4(A_r(u_0,v_0,w_0))} \frac{|K(\xi_1,\xi_1\eta_1)K(\xi_2,\xi_2\eta_2)|^{\frac14}}{|JV(\xi_1,\xi_2,\eta_1,\eta_2)|} |\xi_1\xi_2|dz\leq C,
\end{equation}
with $C$ independent of $k,l,a, u_0,v_0$ and $w_0. $ Here $\xi_1,\xi_1,\eta_1,\eta_2$ are uniquely defined in terms of $z,$ by the bijectivity of $V_4.$ 

We observe that, by \eqref{eq*13}, for $(\xi_1,\xi_2,\eta_1,\eta_2)\in A= A^s_{m,n,p,j}(k,l,a),$ using the fact that $1<\lambda<2, 1<\alpha<2,$
\begin{equation}
\label{eq52} c2^{-k}|\eta_2-\eta_1|<|JV(\xi_1,\xi_2,\eta_1,\eta_2)|<C2^{-k}|\eta_2-\eta_1|,
\end{equation}
where $c,C$ depend only on $d$ and $\epsilon.$  

Also, $$\frac{1}{\lambda^{2(d-2)}}2^{-k-1} < |K(\xi_1,\xi_1\eta_1)|=|\xi_1|^{2(d-2)}|K(1,\eta_1)|<\lambda^{4(d-2)}2^{-k+2}$$
 on the region of integration and thus we need only show that 
 \begin{equation}\label{eq*40}
 \int_{V_4(A_r(u_0,v_0,w_0))} \frac{2^{\frac{k}{2}}}{|\eta_2-\eta_1|}dz\leq C,
 \end{equation}
 with $C$ independent of $k,l,a,u_0,v_0$ and $w_0.$  
 
 We now define $L:V_4(A_r(u_0,v_0,w_0))\rightarrow A_r(u_0,v_0,w_0)$ 
 by $L(z)=V_4^{-1}(z).$ Then $V(L(z))=(u_0,v_0,w_0,z)$ 
 and the chain rule gives us $DV(L(z))L'(z)=(0,0,0,1).$ Routine
 calculations give (with $L = (L_1, L_2, L_3, L_4)$)
 \begin{eqnarray}
 \nonumber L_3'(z) & = & \frac{1}{JV(L(z))}\{dL_2(z)[L_2(z)^{d-1}\phi(1,L_4(z))-L_1(z)^{d-1}\phi(1,L_3(z))]\\
 & & \qquad\qquad -L_2(z)^d\phi'(1,L_4(z))(L_4(z)-L_3(z))\} \label{eq*51}\\
\nonumber  L_4'(z) & = & \frac{1}{JV(L(z))}\{-dL_1(z)[L_2(z)^{d-1}\phi(1,L_4(z))-L_1(z)^{d-1}\phi(1,L_3(z))]\\
 & & \qquad \qquad + L_1(z)^d\phi'(1,L_3(z))(L_4(z)-L_3(z))\}.\label{eq*52}
 \end{eqnarray}

  We now define 
 \begin{eqnarray*}
  & & B_1  = \\
  &   & \{(\xi_1,\xi_2,\eta_1,\eta_2) \mid |\xi_2^d(\phi(1,\eta_2)-\phi(1,\eta_1))|\geq 2|(\xi_2^d-\xi_1^d)\phi(1,\eta_1)|\}\\
  &  & \qquad \cap  A_r(u_0,v_0,w_0),\\
 &  &  B_2   = \\
 &    & \{(\xi_1,\xi_2,\eta_1,\eta_2)\mid |\xi_2^d(\phi(1,\eta_2)-\phi(1,\eta_1))|\leq \frac12|(\xi_2^d-\xi_1^d)\phi(1,\eta_1)|\}\\
 & &  \qquad \cap A_r(u_0,v_0,w_0),\\
 & &  B_3   = 
   \{(\xi_1,\xi_2,\eta_1,\eta_2)\mid \frac12< \frac{|\xi_2^d(\phi(1,\eta_2)-\phi(1,\eta_1))|}{|(\xi_2^d-\xi_1^d)\phi(1,\eta_1)|}< 2\}\\
   &  &\qquad  \cap A_r(u_0,v_0,w_0).
 \end{eqnarray*}
 
 Then we observe that $(\xi_1,\xi_2,\eta_1,\eta_2)\in B_1\Longrightarrow $ 
 \begin{eqnarray*}
V_4(\xi_1,\xi_2,\eta_1,\eta_2)=|z| & = &  |(\xi_2^d-\xi_1^d)\phi(1,\eta_1)+\xi_2^d(\phi(1,\eta_2)-\phi(1,\eta_1))|\\
 &\geq & \frac12|\xi_2^d (\phi(1,\eta_2)-\phi(1,\eta_1))|\geq |(\xi_2^d-\xi_1^d)\phi(1,\eta_1)|
 \end{eqnarray*}
 and 
 $(\xi_1,\xi_2,\eta_1,\eta_2)\in B_2 \Longrightarrow$
$$ V_4(\xi_1,\xi_2,\eta_1,\eta_2)= |z|\geq \frac12|(\xi_2^d-\xi_1^d)\phi(1,\eta_1)|\geq |\xi_2^d(\phi(1,\eta_2)-\phi(1,\eta_1))|.$$ 
 
 Moreover, if $(\xi_2^d-\xi_1^d)\phi(1,\eta_1)$ and $\xi_2^d(\phi(\eta_2)-\phi(\eta_1))$ have the same sign then $|z|\geq \max\{|\xi_2^d(\phi(1,\eta_2)-\phi(1,\eta_1))|,|(\xi_2^d-\xi_1^d)\phi(\eta_1)|\}.$ 
 So if we define
 $$
B_0  =  \{(\xi_1,\xi_2,\eta_1,\eta_2)\in A_r(u_0,v_0,w_0) \mid(\xi_2^d-\xi_1^d)\phi(1,\eta_1)\xi_2^d (\phi(1,\eta_2)-\phi(1,\eta_1))\geq 0\}$$
 then
 \begin{equation}\label{eq17}
z\in V_4(B_0\cup B_1\cup B_2) \Longrightarrow |z|\geq \frac12\max\{|\xi_2^d(\phi(1,\eta_2)-\phi(1,\eta_1))|,|\phi(1,\eta_1)(\xi_2^d-\xi_1^d)|\}.
 \end{equation}
We now have the disjoint union
$$A_r(u_0,v_0,w_0)=B_1\cup B_2\cup (B_3\cap B_0)\cup (B_3\cap B_0^c)$$
and hence
\begin{eqnarray}
   \int_{V_4(A_r(u_0,v_0,w_0))}\frac{2^{\frac{k}{2}}}{\eta_2-\eta_1}dz 
 & = &\int_{V_4(B_{1})}\frac{ 2^{\frac{k}{2}} }{\eta_2-\eta_1} dz + \int_{V_4(B_2)}\frac{ 2^{\frac{k}{2}} }{\eta_2-\eta_1}dz\nonumber  \\
 & &  + \int_{V_4(B_3\cap B_0)}\frac{ 2^{\frac{k}{2}} }{\eta_2-\eta_1}dz +\int_{V_4(B_3\cap B_0^c)}\frac{ 2^{\frac{k}{2}} }{\eta_2-\eta_1}dz\label{eq300} .
  \end{eqnarray}

  In order to be able to bound these integrals we would like to know that we are integrating
  over an interval in $z.$ That we are able to reduce to this case is a consequence of Milnor's theorem (Theorem \ref{milnor} and Corollary \ref{alggeom}).  We note that each of our sets $B_1, B_2, B_3\cap B_0, B_3\cap B_0^c$ is defined by a
 bounded number of polynomial inequalities, where the bound is
 independent of $u_0,v_0,$ $w_0$, and the individual polynomials
 (which have uniformly bounded degrees), and depends only on $d.$ 
By Corollary \ref{alggeom} it follows that the number of connected components of each of $B_0, B_1, B_3\cap B_0, B_3\cap B_0^c$  is bounded by some constant depending only on $d$. We consider $B_1,$ which we can write as 
$B_1=\bigcup_{m=1}^{m_0} B_{1,m}$ where each $B_{1,m}$ is connected and $m_0$ depends only on $d.$ Then, for each $m,$ $V_4(B_{1,m})$ is the continuous image of a connected set, hence is an interval. Thus $V_4(B_{1})$ is a union of  boundedly many intervals, with bound depending only $d$, and hence it suffices to work with one of them. Indeed, by slight abuse of notation, we may assume that $V_4(B_1)$ itself is an interval. We can argue similarly for $V_4(B_2), V_4(B_3\cap B_0), V_4(B_3\cap B_0^c).$ 

 We will be able to deal with the first three integrals on the
 right-hand side of \eqref{eq300} using the lower bound \eqref{eq17}. First we consider the integral over $V_4(B_{3}\cap B_0^c).$  Here a finer decomposition is needed. 
 We let 
 \begin{eqnarray*}
   B_{3}^i & =&
 \left\{(\xi_1,\xi_2,\eta_1,\eta_2)\mid 1+ 2^{-i-1}\leq \left|\frac{\xi_2^d(\phi(1,\eta_2)-\phi(1,\eta_1))}{(\xi_2^d-\xi_1^d)\phi(1,\eta_1)}\right| \leq 1+ 2^{-i}\right\}\\
 & & \cap A_r(u_0,v_0,w_0).
\end{eqnarray*}
 and
\begin{eqnarray*} 
  \tilde{B}_{3}^i & = & 
\left\{(\xi_1,\xi_2,\eta_1,\eta_2) \mid 1+ 2^{-i-1}\leq \left|\frac{(\xi_2^d-\xi_1^d)\phi(1,\eta_1)}{\xi_2^d(\phi(1,\eta_2)-\phi(1,\eta_1))}\right| \leq  1+ 2^{-i}\right\}\\
& & \cap A_r(u_0,v_0,w_0).
\end{eqnarray*}

Then $B_{3}\cap B_0^c=\left[\bigcup_{i=0}^\infty B_{3}^i\cup\bigcup_{i=1}^\infty \tilde{B}_{3}^i\right]\cap B_0^c.$ 

 Now
  \begin{eqnarray*}
 \int_{V_4(B_3\cap B_0^c)}\frac{2^{\frac{k}{2}}}{\eta_2-\eta_1}dz =
 \sum_{i=0}^{\infty} \int_{V_4(B_3^i\cap B_0^c)}\frac{2^{\frac{k}{2}}}{\eta_2-\eta_1}dz + 
 \sum_{i=0}^{\infty}\int_{V_4(\tilde{B}_3^i\cap B_0^c)}\frac{2^{\frac{k}{2}}}{\eta_2-\eta_1}dz.
    \end{eqnarray*}
Since each of the sets $B_{3}^i\cap B_0^c,\tilde{B}_{3}^i\cap B_0^c$ is defined by
a bounded number of polynomial
inequalities the same argument as before gives us that $V_4(B_{3}^i\cap B_0^c)$ and $V_4(\tilde{B}_{3}^i\cap B_0^c)$ can
be written as a finite union of intervals whose total number is
bounded, with bound depending only on $d$. As before, this allows us to assume, by slight abuse of notation, that $V_4(B_{ 3}^i\cap B_0^c)$ is an interval in z, say $I^i_3=[a_3^i,b_3^i].$ 
Similarly we may assume that $V_4(\tilde{B}^i_3\cap B_0^c)$ is an interval, $\tilde{I}_3^i.$  
 
 Then   
 \begin{eqnarray}
\int_{I_3^i}\frac{2^{\frac{k}{2}}}{\eta_2-\eta_1} dz & = & 2^{\frac{k}{2}} \int_{I_3^i} \left|\frac{\xi_2^d(\phi(1,\eta_2)-\phi(1,\eta_1))}{\eta_2-\eta_1}\right|\frac{dz}{|\xi_2^d(\phi(1,\eta_2)-\phi(1,\eta_1))|}\nonumber \\
 & \leq & 2^{\frac{k}{2}} \lambda^{2d} \alpha^{-l+2}\int_{I_3^i}\frac{dz}{|\xi_2^d(\phi(1,\eta_2)-\phi(1,\eta_1))|}\nonumber \\
 & \leq & 2^{2+2d} 2^{\frac{k}{2}} \alpha^{-l} \int_{I_3^i}\frac{dz}{|\xi_2^d(\phi(1,\eta_2)-\phi(1,\eta_1))|}\label{eq100}.
 \end{eqnarray}
 There is a similar inequality for $\tilde{I}_{3}^i.$ 
 We recall that, by \eqref{lowerbound}, $\frac{K(1,\eta)}{\phi'(1,\eta)^2}\geq \epsilon $ on $Z_j,$ and hence
 $2^{-k}\alpha^{2l} \geq \frac{\epsilon}{4\alpha^2}>\frac{\epsilon}{16}. $ Using this in \eqref{eq100}, together with its analogue for $\tilde{I}_3^i,$  we see that it suffices to show that 
\begin{equation}\label{eq101}
\sum_{i=0}^{\infty} \int_{I_3^i}\frac{dz}{|\xi_2^d(\phi(1,\eta_2)-\phi(1,\eta_1))|}
+ \sum_{i=0}^\infty \int_{\tilde{I}_3^i}\frac{dz}{|\xi_2^d (\phi(1,\eta_2)-\phi(1,\eta_1))|}\leq C,
\end{equation}
for some constant $C$ depending only on $d$ and $\epsilon.$

Now if $(\xi_1,\xi_2,\eta_1,\eta_2)\in B_{3}\cap B_0^c$ then $(\xi_1,\xi_2,\eta_1,\eta_2)\in B_3$ and 
$\xi_2^d(\phi(1,\eta_2)-\phi(1,\eta)),$ $ (\xi_2^d-\xi_1^d)\phi(1,\eta_1)$ have opposite signs. 
We suppose first that $(\xi_1,\xi_2,\eta_1,\eta_2)\in B_{3}^i\cap B_0^c.$ Then 
\begin{eqnarray*}
\frac{2^{-i-1}}{1+2^{-i}}|\xi_2^d(\phi(1,\eta_2)-\phi(1,\eta_1))|&  \leq &  2^{-i-1}|(\xi_2^d-\xi_1^d)\phi(1,\eta_1)| \\
& \leq & |z|\\
& \leq & 2^{-i}|(\xi_2^d-\xi_1^d)\phi(1,\eta_1)|\\
& \leq &  \frac{2^{-i}}{1+2^{-i-1}}|\xi_2^d (\phi_2(1,\eta_2)-\phi_1(1,\eta_1))|.
\end{eqnarray*}
 
 Similarly, if $(\xi_1,\xi_2,\eta_1,\eta_2)\in \tilde{B}_{3}^i\cap B_0^c,$ then 
 $$\frac{2^{-i-1}}{1+2^{-i}}|(\xi_2^d-\xi_1^d)\phi(1,\eta_1)| \leq    |z|
\leq \frac{2^{-i}}{1+2^{-i-1}}|(\xi_2^d-\xi_1^d)\phi(1,\eta_1)|.$$

 Thus, for $z\in I_3^i\cup \tilde{I}_3^i$,
 \begin{equation}\label{eq102}
 2^{-i-2}|\xi_2^d(\phi(1,\eta_2)-\phi(1,\eta_1))|\leq |z| \leq 2^{-i+1}|\xi_2^d(\phi(1,\eta_2)-\phi(1,\eta_1))|.
 \end{equation}

  We will now show that there exists $i_0=i_0(\epsilon,d)$ such that $i\geq i_0\Longrightarrow \left|\frac{z}{z'}\right|\leq C, \forall
 z,z'\in I_3^i$ and $\forall z,z'\in \tilde{I}_3^i,$ with $C$ depending only on $d$ and $\epsilon.$

By \eqref{eq102}, for $ z,z'\in  I_3^i$ we have
\begin{eqnarray}
& & \left|\frac{z}{z'}\right|\nonumber \\
& \leq & 8 \left|\frac{\xi_2^d(\phi(1,\eta_2)-\phi(1,\eta_1))(z)}{\xi_2^d(\phi(1,\eta_2)-\phi(1,\eta_1))(z')}\right|\nonumber\\
& \leq & 8\left\{|\xi_2^d (z)|\frac{|(\phi(1,\eta_2)-\phi(1,\eta_1))(z)-(\phi(1,\eta_2)-\phi(1,\eta_1))(z')|}{|\xi_2^d(\phi(1,\eta_2)-\phi(1,\eta_1))(z')|} +\left|\frac{\xi_2^d(z)}{\xi_2^d(z')}\right|\right\}\nonumber\\
& \leq & 8 \left\{\lambda^{2d}2^{-i+1}\left|\frac{z-z'}{z'}\right|\cdot\left| \frac{d}{d z}(\phi(1,\eta_2)-\phi(1,\eta_1))(\theta)\right|+\lambda^{3d}\right\},\label{eq104}
\end{eqnarray}
for some $\theta$ between $z$ and $z'.$

We now use \eqref{eq*51} and \eqref{eq*52} to give
\begin{eqnarray*}
& & \left|\phi'(1,\eta_2)\frac{d\eta_2}{d z}-\phi'(1,\eta_1)\frac{d \eta_1}{d z}\right|\\
& = &\left|\left\{-d[\xi_2^{d-1}\phi(1,\eta_2)-\xi_1^{d-1}\phi(1,\eta_1)](\xi_1\phi'(1,\eta_2)+\xi_2\phi'(1,\eta_1)) \right.\right.\\
& &\qquad  \left.\left. +(\xi_1^d+\xi_2^d)\phi'(1,\eta_1)\phi'(1,\eta_2)(\eta_2-\eta_1)\right\}\right|\slash |JV(\xi_1,\xi_2,\eta_2,\eta_2)|\\
& \leq & \frac{1}{|JV(\xi_1,\xi_2,\eta_1,\eta_2)|}\left\{d|(\xi_2^{d-1}-\xi_1^{d-1})\phi(1,\eta_1)|\cdot |\xi_1\phi'(1,\eta_2)+\xi_2\phi'(1,\eta_1)| \right.\\
& & \quad\  + d|\xi_2^{d-1}|\cdot |\phi(1,\eta_2)-\phi(1,\eta_1)|\cdot |\xi_1\phi'(1,\eta_2)+\xi_2\phi'(1,\eta_1)|\\
& & \quad \ \left. + (|\xi_1^d|+|\xi_2|^d)\cdot |\phi'(1,\eta_1)\phi'(1,\eta_2)|\cdot |\eta_2-\eta_1|\right\}.
\end{eqnarray*} 

 For $(\xi_1,\xi_2,\eta_1,\eta_2)\in B_{3}\cap B_0^c$ we have
 \begin{eqnarray*}
 |(\xi_2^{d-1}-\xi_1^{d-1})\phi(1,\eta_1)| & = & |(\xi_2^d-\xi_1^d)\phi(1,\eta_1)|\cdot\left|\frac{\xi_2^{d-1}-\xi_1^{d-1}}{\xi_2^d-\xi_1^d}\right|\\
 & \leq & 2\sqrt{2}\frac{d-1}{d}\lambda |\xi_2^d (\phi(1,\eta_2)-\phi(1,\eta_1))|.
 \end{eqnarray*}
 
 Then
 \begin{eqnarray}
 & & \left|\phi'(1,\eta_2)\frac{d \eta_2}{d z}-\phi'(\eta_1)\frac{d \eta_1}{dz}\right|\label{eq105}\leq\\
 &  &\left\{2 \left(2\sqrt{2}(d-1)+\frac{d\sqrt{2}}{\lambda}\right)\frac{|\xi_2^d(\phi(1,\eta_2)-\phi(1,\eta_1))|}{|\eta_2-\eta_1|} |\eta_2-\eta_1| \lambda^{3}\alpha^{-l+2}\right. \nonumber\\
 & & \left.\qquad+2\lambda^{2d}\alpha^{2(-l+2)} |\eta_2-\eta_1|\right\}\slash |JV(\xi_1,\xi_2,\eta_1,\eta_2)| \nonumber\\
 & &\leq  2^k\alpha^{-2l}(d-1)\alpha^{10}\lambda^{4d} 2\left[\left(2\sqrt{2}(d-1)+\frac{d\sqrt{2}}{\lambda}\right)\lambda^3+1\right] (\text{using}\ \eqref{eq*13})\nonumber\\
 & &\leq  \frac{8}{\epsilon}\alpha^{12}\lambda^{4d}(d-1)\left[\left(2\sqrt{2}(d-1)+\frac{d\sqrt{2}}{\lambda}\right)\lambda^3+1\right]\nonumber  (\text{since } 2^{-k}\alpha^{2l}\geq\frac{\epsilon}{4\alpha^2}) \nonumber\\
 & & \leq \frac{2^{21+4d}}{\epsilon}(d-1)^2 (\text{since } 1<\lambda<2, 1<\alpha<2)\nonumber \\
 & & = c_{\epsilon,d},\nonumber
 \end{eqnarray}
where $c_{\epsilon,d}$  is as defined in \eqref{cepsilond}.

 We now substitute \eqref{eq105} in \eqref{eq104} and obtain
  $$\left|\frac{z}{z'}\right|\leq 8\left\{\lambda^{2d}2^{-i+1}\left(\left|\frac{z}{z'}\right|+1\right)c_{\epsilon,d} +\lambda^{3d}\right\}$$
and hence
$$\left|\frac{z}{z'}\right|\leq \frac{16\lambda^{2d}2^{-i}c_{\epsilon,d}+8\lambda^{3d}}{1-16\lambda^{2d}2^{-i}c_{\epsilon,d}},$$
provided the denominator on the right-hand side is positive.
We now choose $i_0$ to be the first non-negative integer $i$ such that $16\lambda^{2d}2^{-i} c_{\epsilon,d}\leq \frac12.$ Hence
$$\frac14 < 16\lambda^{2d}2^{-i_0}c_{\epsilon,d}\leq \frac12$$
and $16\lambda^{2d}2^{-i}c_{\epsilon,d}\leq \frac12, \forall i\geq i_0.$ 
 We note that $i_0$ is independent of $k$ and $l;$ moreover
 $$2^{5}c_{\epsilon_d}\leq 2^{i_0}<2^{2d+6}c_{\epsilon,d}.$$
    Hence 
$$\left|\frac{z}{z'}\right|\leq \frac{\frac12+8\lambda^{3d}}{\frac12}=1+16\lambda^{3d},\forall z,z'\in I_3^i, i\geq i_0.$$
Exactly the same  argument for $z,z'\in \tilde{I}_3^i$ gives
$$\left|\frac{z}{z'}\right|\leq 1+16\lambda^{3d},\forall z,z'\in \tilde{I}_3^i,\ i\geq i_0.$$
Then
\begin{eqnarray*}
& & \sum_{i=i_0}^{\infty}\int_{I_3^i}\frac{dz}{|\xi_2^d(\phi(1,\eta_2)-\phi(1,\eta_1))|} + 
\sum_{i=i_0}^{\infty}\int_{\tilde{I}_3^i}\frac{dz}{|\xi_2^d(\phi(1,\eta_2)-\phi(1,\eta_1))|} \\
& \leq & \sum_{i=i_0}^{\infty} 2^{-i+1}\int_{I_3^i}\frac{dz}{|z|} +\sum_{i=i_0}^\infty 2^{-i+1}\int_{\tilde{I}_3^i}\frac{dz}{|z|} \\
& \leq & 8(1+16\lambda^{3d})\leq 8(1+2^{3d+4}).
\end{eqnarray*}

To obtain \eqref{eq101}  it remains to show that we can bound  the sum from $0$ to $i_0.$ We will deal with these terms simultaneously with the integral over $V_4(B_1)\cup V_4(B_2)\cup V_4(B_3\cap B_0).$  We have
\begin{eqnarray*}
& & \bigcup_{i=0}^{i_0-1}(B_{3}^i\cup \tilde{B}_{3}^i)= \\
&  & \left[ \{(\xi_1,\xi_2,\eta_1,\eta_2)\mid1+2^{-i_0}\leq \frac{|\xi_2^d(\phi(1,\eta_2)-\phi(1,\eta_1))|}{|(\xi_2^d-\xi_1^d)\phi(1,\eta_1)|}\leq 2\}\cap A_r(u_0,v_0,w_0)\right]\\
& &  \cup
\left[\{(\xi_1,\xi_2,\eta_1,\eta_2)\mid1+2^{-i_0}\leq \frac{|(\xi_2^d-\xi_1^d)\phi(1,\eta_1)|}{|\xi_2^d(\phi(1,\eta_2)-\phi(1,\eta_1))|}\leq 2\}\cap A_r(u_0,v_0,w_0)\right].
\end{eqnarray*}
Then, for $(\xi_1,\xi_2,\eta_1,\eta_2)\in  \bigcup_{i=0}^{i_0-1}(B_{3}^i\cup \tilde{B}_{3}^i) $ we have
$$|z|\geq 2^{-i_0-1}\max \{|(\xi_2^d-\xi_1^d)\phi(1,\eta_1)|,|\xi_2^d(\phi(1,\eta_2)-\phi(1,\eta_1))|\}.$$
We let $B_4=\{(\xi_1,\xi_2,\eta_1,\eta_2)\mid |z|\geq 2^{-i_0-1}\max \{|(\xi_2^d-\xi_1^d)\phi(1,\eta_1)|,|\xi_2^d(\phi(1,\eta_2)-\phi(1,\eta_1))|\}.$ 
Then the proof will be complete once we show that 
\begin{eqnarray}
& & \int_{V_4(B_1)}  \frac{2^{\frac{k}{2}}}{\eta_2-\eta_1}dz + \int_{V_4(B_2)}  \frac{2^{\frac{k}{2}}}{\eta_2-\eta_1}dz + \int_{V_4(B_3\cap B_0)}  \frac{2^{\frac{k}{2}}}{\eta_2-\eta_1}dz\nonumber \\
& & \qquad\qquad 
+ \int_{V_4(B_3\cap B_0^c\cap B_4)}  \frac{2^{\frac{k}{2}}}{\eta_2-\eta_1}dz\leq C,\label{eq106}
\end{eqnarray}
with $C$ independent of $k$ and $l.$ 
By our earlier comments we may assume that, in each of the four
integrals on the left-hand side,  the range of integration in $z$ is an interval. Moreover, by splitting into two intervals if necessary, we may assume that the interval is of 
the form $[a,b]$ with  either $b>a\geq 0$ or $a<b\leq 0. $ Thus
\eqref{eq106} follows once we show that 
\begin{equation}\label{eq107}
\int_a^b \frac{2^{\frac{k}{2}}}{\eta_2-\eta_1}dz\leq C
\end{equation}
where, for $z\in [a,b],$ we have
\begin{equation}\label{eq41}
|z|\geq 2^{-i_0-1}\max \{|(\xi_2^d-\xi_1^d)\phi(1,\eta_1)|,|\xi_2^d(\phi(1,\eta_2)-\phi(1,\eta_1))|\}.
\end{equation}
(We note that a stronger inequality, \eqref{eq17}, holds on $V_4(B_1)\cup V_4(B_2)\cup V_4(B_3\cap B_0)$.)   We now need the following lemma. First, we define
\begin{eqnarray*}
& & S(\xi_1,\xi_2,\eta_1,\eta_2)=\\
& & \frac{(\eta_2-\eta_1)(\xi_1^d \phi'(1,\eta_1)+\xi_2^d\phi'(1,\eta_2))-d(\xi_1+\xi_2)(\xi_2^{d-1}\phi(1,\eta_2)-\xi_1^{d-1}\phi(1,\eta_1))}{(\xi_2^d-\xi_1^d)\phi(1,\eta_1)+ (\phi(1,\eta_2)-\phi(1,\eta_1))\xi_2^d}.
\end{eqnarray*}

\begin{lemma}\label{lemma4}
Let $A_{0}=\{(\xi_1,\xi_2,\eta_1,\eta_2)\in A_r(u_0,v_0,w_0) \mid |(\xi_2^d-\xi_1^d)\phi(1,\eta_1)+ (\phi(1,\eta_2)-\phi(1,\eta_1))\xi_2^d|\geq 2^{-i_0-1} \max\{|(\xi_2^d-\xi_1^d)\phi(1,\eta_1)|,|(\phi(1,\eta_2)-\phi(1,\eta_1))\xi_2^d|\}.$  
Then 
$$-3(d-1)\leq S(\xi_1,\xi_2,\eta_1,\eta_2)\leq -(d-1)$$
for all $(\xi_1,\xi_2,\eta_1,\eta_2)\in A_{0}.$ 
\end{lemma}

  \begin{proof}
  We have
  \begin{eqnarray*}
  & & S(\xi_1,\xi_2,\eta_1,\eta_2) \\
  & = & \frac{(\eta_2-\eta_1)(\xi_1^d\phi'(1,\eta_1)+\xi_2^d\phi'(1,\eta_2))-d(\xi_1+\xi_2)(\xi_2^{d-1}\phi(1,\eta_2)-\xi_1^{d-1}\phi(1,\eta_1))}{(\xi_2^d-\xi_1^d)\phi(1,\eta_1)+ (\phi(1,\eta_2)-\phi(1,\eta_1))\xi_2^d}\\
  & = & \frac{(\eta_2-\eta_1)(\xi_1^d\phi'(1,\eta_1)+\xi_2^d\phi'(1,\eta_2))-\frac{d(\xi_1+\xi_2)}{\xi_2}\xi_1^{d-1}(\xi_1-\xi_2)\phi(1,\eta_1)}{(\xi_2^d-\xi_1^d)\phi(1,\eta_1)+ (\phi(1,\eta_2)-\phi(1,\eta_1))\xi_2^d}\\
  & & -d\left(1+\frac{\xi_1}{\xi_2}\right).
  \end{eqnarray*}
  Hence 
  \begin{eqnarray*}
  & & S(\xi_1,\xi_2,\eta_1,\eta_2)+2(d-1) \\
  & = & \frac{(\eta_2-\eta_1)(\xi_1^d\phi'(1,\eta_1)+\xi_2^d\phi'(1,\eta_2))-2\xi_2^d(\phi(1,\eta_2)-\phi(1,\eta_1))}{(\xi_2^d-\xi_1^d)\phi(1,\eta_1)+ (\phi(1,\eta_2)-\phi(1,\eta_1))\xi_2^d}\\ 
  & & -\frac{2\phi(1,\eta_1)(\xi_2^d-\xi_1^d)-d\frac{\xi_1^{d-1}}{\xi_2}(\xi_2^2-\xi_1^2)\phi(1,\eta_1)}{(\xi_2^d-\xi_1^d)\phi(1,\eta_1)+ (\phi(1,\eta_2)-\phi(1,\eta_1))\xi_2^d}+d\left(1-\frac{\xi_1}{\xi_2}\right).
  \end{eqnarray*}
  Then
  \begin{eqnarray*}
  & & |S(\xi_1,\xi_2,\eta_1,\eta_2)+2(d-1)|\\
  & \leq & \frac{|(\eta_2-\eta_1)(\xi_1^d\phi'(1,\eta_1)+\xi_2^d\phi'(1,\eta_2))-2\xi_2^d(\phi(1,\eta_2)-\phi(1,\eta_1))|}{|(\xi_2^d-\xi_1^d)\phi(1,\eta_1)+ (\phi(1,\eta_2)-\phi(1,\eta_1))\xi_2^d|}\\
  & & + \left|2(\xi_2^d-\xi_1^d)-d\frac{\xi_1^{d-1}}{\xi_2}(\xi_2^2-\xi_1^2)\right|\frac{|\phi(1,\eta_1)|}{|(\xi_2^d-\xi_1^d)\phi(1,\eta_1)+ (\phi(1,\eta_2)-\phi(1,\eta_1))\xi_2^d|} \\
  & & + d\left| 1-\frac{\xi_1}{\xi_2}\right|\\
  & = & I+II+III.
  \end{eqnarray*}
  We will now show that for our choice of $\lambda$ and $\alpha $ we have $I+II+III\leq (d-1).$ If so, then
  the result is immediate. We begin by estimating each of the terms separately. We have
  \begin{eqnarray*}
  I & = & \left| \frac{(\eta_2-\eta_1)(\xi_1^d\phi'(1,\eta_1)+\xi_2^d\phi'(1,\eta_2))-2\xi_2^d(\phi(1,\eta_2)-\phi(1,\eta_1))}{(\xi_2^d-\xi_1^d)\phi(1,\eta_1)+ (\phi(1,\eta_2)-\phi(1,\eta_1))\xi_2^d}\right|\\
 & \leq & 2^{i_0+1}\left|\frac{  (\eta_2-\eta_1)(\xi_1^d\phi'(1,\eta_1)+\xi_2^d\phi'(1,\eta_2))-2\xi_2^d(\phi(1,\eta_2)-\phi(1,\eta_1))}{\xi_2^d (\phi(1,\eta_2)-\phi(1,\eta_1))}\right|\\
  & = & 2^{i_0+1}
  \left|\frac{\frac{\xi_1^d}{\xi_2^d}\phi'(1,\eta_1)+\phi'(1,\eta_2)}{\frac{\phi(1,\eta_2)-\phi(1,\eta_1)}{\eta_2-\eta_1}}-2\right|\\
  & \leq & 2^{i_0+1}
  \left\{ \left|\frac{\phi'(1,\eta_1)+\phi'(1,\eta_2)}{\frac{\phi(\eta_2)-\phi(1,\eta_1)}{\eta_2-\eta_1}}-2\right|+\left|\frac{\xi_1^d}{\xi_2^d}-1\right|\frac{|\phi'(1,\eta_1)|}{\left|\frac{\phi(1,\eta_2)-\phi(1,\eta_1)}{\eta_2-\eta_1}\right|}\right\}.
  \end{eqnarray*}
  Now for $(\xi_1,\xi_2,\eta_1,\eta_2)\in A_r(u_0,v_0,w_0),$ we have
  $$-\frac{(\lambda^{3d}-1)}{\lambda^{3d}}=\frac{1}{\lambda^{3d}}-1\leq \left(\frac{\xi_1}{\xi_2}\right)^d-1\leq \lambda^{3d}-1$$ and so
  $$\left|\left( \frac{\xi_1}{\xi_2}\right)^d-1\right|\leq \lambda^{3d}-1.$$
  Also $\frac{\phi(1,\eta_2)-\phi(1,\eta_1)}{\eta_2-\eta_1}=\phi'(1,\eta_4)$ for some $\eta_4\in (\eta_1,\eta_2)$ and hence 
  $$\frac{2}{\alpha^3}\leq \frac{\phi'(1,\eta_1)+\phi'(1,\eta_2)}{\frac{\phi(1,\eta_2)-\phi(1,\eta_1)}{\eta_2-\eta_1}}\leq 2\alpha^3$$ and then
  $$\left|   \frac{\phi'(1,\eta_1)+\phi'(1,\eta_2)}{\frac{\phi(1,\eta_2)-\phi(1,\eta_1)}{\eta_2-\eta_1}}-2\right|\leq 2(\alpha^3-1).$$
  Then 
  \begin{equation}\label{eq16a}
  I \leq 2^{i_0+1}[2(\alpha^3-1)+(\lambda^{3d}-1)\alpha^3].
  \end{equation}
  
  Next we have 
  $$
  II  \leq  2^{i_0+1}
  \frac{|2(\xi_2^d-\xi_1^d)-d\frac{\xi_1^{d-1}}{\xi_2}(\xi_2^2-\xi_1^2)|}{|(\xi_2^d-\xi_1^d)|}=  2^{i_0+1}\left|2-\frac{d\xi_1^{d-1}(\xi_2^2-\xi_1^2)}{\xi_2(\xi_2^d-\xi_1^d)}\right|.
  $$
  Now
  $$  d\frac{\xi_1^{d-1}}{\xi_2}\frac{\xi_2^2-\xi_1^2}{\xi_2^d-\xi_1^d}  =  d
\left(\frac{\xi_1}{\xi_2}\right)^{d-1}
  \frac{\left(1-\left(\frac{\xi_1}{\xi_2}\right)^2\right)}{\left(1-\left(\frac{\xi_1}{\xi_2}\right)^d\right)}
   =  2\left(\frac{\xi_1}{\xi_2}\right)^{d-1}\frac{1}{\theta^{d-2}}$$
for some $\theta$ between $\frac{\xi_1}{\xi_2}$ and $1.$ 
  We recall that $\frac{1}{\lambda^3}<\frac{\xi_1}{\xi_2}<\lambda^3$ and so we consider two cases. If $\frac{1}{\lambda^3}<\frac{\xi_1}{\xi_2}<\theta<1<\lambda^3,$ then 
  $$\frac{2}{\lambda^{3(d-1)}}<2\left(\frac{\xi_1}{\xi_2}\right)^{d-1}\frac{1}{\theta^{d-2}}<2.$$
  If $1<\theta<\frac{\xi_1}{\xi_2}<\lambda^3,$ then
  $$2< 2\left(\frac{\xi_1}{\xi_2}\right)^{d-1}\frac{1}{\theta^{d-2}}<2\lambda^{3(d-1)}.$$ 
  Thus in all cases we have 
  $$\frac{2}{\lambda^{3(d-1)}}<  d\frac{{\xi_1}^{d-1}}{\xi_2}\frac{\xi_2^2-\xi_1^2}{\xi_2^d-\xi_1^d}<2\lambda^{3(d-1)}.$$
  Hence 
  $$-2(\lambda^{3(d-1)}-1)  <  2-\frac{d\xi_1^{d-1}(\xi_2^2-\xi_1^2)}{\xi_2(\xi_2^d-\xi_1^d)}<2(\lambda^{3(d-1)}-1).$$
  Thus we obtain 
  \begin{equation}\label{eq16b}
  II\leq 2^{i_0+2}(\lambda^{3(d-1)}-1).
  \end{equation}
  
  Finally, we have the easy estimate
  \begin{equation}\label{eq16c}
  III=d\left|1-\frac{\xi_1}{\xi_2}\right|\leq d(\lambda^3-1).
  \end{equation}
  
 We recall that 
$ 2^{i_0} <2^{2d+6}c_{\epsilon,d}.$ Also, 
   a routine calculation shows that, for $d\geq 2,$ $\lambda^3-1\leq \lambda^{3(d-1)}-1\leq \lambda^{3d}-1\leq \delta,$ and so \eqref{eq16a}, \eqref{lambda}, \eqref{alpha} and \eqref{delta} together  give us
  \begin{eqnarray*}
  I & \leq & 2^{i_0+1}\left[2\delta+8\delta\right]\\
  & \leq & 2^{2d+7}c_{\epsilon,d}10\delta\\
    & \leq & 2^{2d+7}\frac{2^{21+4d}}{\epsilon}(d-1)^2 10\frac{\epsilon}{2^{33+6d}(d-1)}\\
    & = & (d-1)\frac{10}{2^5}\\
    & \leq &  \frac{1}{3}(d-1).
  \end{eqnarray*}
  Also, using \eqref{eq16b}, \eqref{lambda}, \eqref{delta} and a similar argument, we have
 $$  II  \leq   {2^{i_0+2}}c_{\epsilon,d} \delta \leq  \frac{1}{2^4} (d-1) \leq \frac13 (d-1)$$
   and  finally, using \eqref{eq16c}, \eqref{lambda} and  \eqref{delta},
$$
   III  \leq d(\lambda^3-1)\leq d\delta\leq \frac{1}{3}(d-1).$$
   Hence we have
   $I+II+III\leq (d-1),$ and so the proof is complete.
  
 \end{proof}

We now use Lemma \ref{lemma4} to prove \eqref{eq107}. We have
\begin{equation}\label{eqH}
-3(d-1)\leq S(\xi_1,\xi_2,\eta_1,\eta_2)\leq -(d-1),
\end{equation}
for all $(\xi_1,\xi_2,\eta_1,\eta_2)\in B_1\cup B_2\cup (B_3\cap B_0) \cup (B_3\cap B_0^c\cap B_4)$.  

We define $F(z)=\eta_2-\eta_1=L_4(z)-L_3(z).$ Then, using 
\eqref{eq*51} and \eqref{eq*52}, we have
 \begin{eqnarray*}
 & & F(z)F'(z)\\
 & = &\frac{L_4(z)-L_3(z)}{JV(L(z))}\{(L_4(z)-L_3(z))[L_1(z)^d\phi'(1,L_3(z))+L_2(z)^d \phi'(1,L_4(z))] \\ 
 & & \qquad\qquad -d(L_1(z)+L_2(z))[L_2(z)^{d-1}\phi(1,L_4(z))-L_1(z)^{d-1}\phi(1,L_3(z))]\},
 \end{eqnarray*}
 which we rewrite as 
  \begin{eqnarray*}
  F(z)F'(z) & = &
  \frac{\eta_2-\eta_1}{JV(\xi_1,\xi_2,\eta_1,\eta_2)}\left\{ (\eta_2-\eta_1)(\xi_1^d \phi'(1,\eta_1)+\xi_2^d\phi'(1,\eta_2))-\right.\\
 & & \qquad\qquad \left. d(\xi_1+\xi_2)(\xi_2^{d-1}\phi(1,\eta_2)-\xi_1^{d-1}\phi(1,\eta_1))\right\}.
 \end{eqnarray*}
 
and so  
$$\frac{F(z)F'(z)}{z}=\frac{\eta_2-\eta_1}{JV(\xi_1,\xi_2,\eta_1,\eta_2)} S(\xi_1,\xi_2,\eta_1,\eta_2).$$
If we now substitute using \eqref{eqH} and \eqref{eq52} we have
\begin{equation}\label{eq20}
-3\frac{(d-1)2^k}{c}\leq \frac{F(z)F'(z)}{z} \leq -\frac{(d-1)2^k}{C}.
\end{equation}
Now, since we are considering the region where $\eta_2> \eta_1,$ we have $F(z)> 0,$ and then from \eqref{eq20} we see that $F'(z)$ is positive for $z< 0$ and negative for $z> 0.$ We also note that $z=0$ is not in our region of integration, by \eqref{eq41}.

Then for $z> 0$ we have
$$-6\frac{(d-1)2^k}{c}z\leq \frac{d}{dz}(F(z)^2)\leq -2\frac{(d-1)2^k}{C}z,$$
and for $z<0$ we have
$$-2\frac{(d-1)2^k}{c}z\leq \frac{d}{dz}(F(z)^2)\leq -6\frac{(d-1)2^k}{C}z$$
 
We take first the case $0\leq a<b.$ Then, on integrating the differential inequalities, using the fact that they hold in an interval in $z$, for $a\leq z\leq b$ we have 
$$-3\frac{(d-1)2^k}{c}(b^2-z^2) \leq F(b)^2-F(z)^2\leq -\frac{(d-1)2^k}{C}(b^2-z^2).$$
Thus
\begin{eqnarray*}
\int_{a}^b\frac{2^{\frac{k}{2}}}{\eta_2-\eta_1} dz & = & 2^{\frac{k}{2}}\int_a^b\frac{dz}{F(z)}\\ 
& \leq &  2^{\frac{k}{2}}\int_{a}^b \frac{dz}{\sqrt{F(b)^2+\frac{(d-1)2^k(b^2-z^2)}{C}}} \\
& \leq & C\int_a^b \frac{dz}{\sqrt{b^2-z^2}} \\
& \leq &  C,
\end{eqnarray*}
with $C$ depending only on $\epsilon$ and $d,$ and independent of $k,l.$ 
This gives \eqref{eq107} as required. A similar argument works in the
case $a<b\leq 0. $ This completes the proof of Proposition \ref{prop2.5}. \bigskip

We now need to consider $j\in J_2.$ We let $\phi_j=\phi\circ T_j^{-1},$ as defined in
Lemma \ref{propZ3}. We define $$Z_{\phi_j}=\{(\xi,\mu)\mid |\mu|<|\xi|, \left|\tan^{-1}\left(\frac{\mu}{\xi}\right)-\tan^{-1}(\eta^0)\right|<\tilde{\epsilon}_j\},$$
in analogy with $Z_j.$ From Lemma \ref{propZ3} we know that $T_j(Z_j)\subset Z_{\phi_j},$ 
that $K_{\phi_j}(1,\eta)$ has no zero other than $\eta^0$ in $\overline{Z_{\phi_j}},$ and $\phi'(1,\eta^0)=0.$ Therefore, Lemma \ref{propZ2} for this $\phi_j$ shows that 
 $|K_{\phi_j}(1,\frac{\mu}{\xi})|\geq \epsilon'|\phi_j'(1,\frac{\mu}{\xi})|^2,$ for all $(\xi,\mu)\in Z_{\phi_j},$ where $\epsilon'$ depends only on $\phi_j$ (and hence only on $\phi$).
  We then have corresponding $\lambda,\alpha$ as in \eqref{lambda}, \eqref{alpha} and we define $\Psi_{0lk}$
and the sets $A_{m,n,p,j}^s$ as before, but with $\phi_j$ replacing $\phi$ and $Z_{\phi_j}$ replacing $Z_j.$ 
   We define
$$\widehat{(\tilde{W}_j g)}(\xi,\eta)= \chi_{T_j(Z_j)}(\xi,\mu)\hat{g}(\xi,\mu).$$
Then we have an analogue of Proposition \ref{prop2.5} for $\phi_j.$ Thus we obtain, as before, for $j\in J_2,$
$$\|U_j (t)\tilde{W}_j g\|_{L^4_{xyt}}\leq C\| g\|_{L^2_{xy}},\ \forall g\in L^2(\R^2), $$
with constant $C$ depending only on $j,\epsilon'$ and $d,$ where $U_j$ corresponds to $\phi_j.$ 
By the parametrization invariance of the restriction property it follows that, for $j\in J_2,$ 
\begin{equation}\label{lasteq}
\|U (t)W_j g\|_{L^4_{xyt}}\leq C\|g\|_{L^2_{xy}},\ \forall g\in L^2(\R^2),
\end{equation}
with constant $C$ depending only on $j,d$ and $\epsilon'.$   
If we now combine \eqref{lasteq} with Corollary \ref{corollary2.6}b) we have
$$\|U(t) W_j g\|_{L^4_{xyt}}\leq C\|g\|_{L^2_{xy}},\ \forall g\in L^2(\R^2), \forall j=1,\dots, R.$$
 By the remark preceding Proposition \ref{prop1}, the proof is complete.\medskip

It remains only to establish our Littlewood-Paley theorem, Proposition \ref{prop2}.

\section{Proof of Proposition \ref{prop2}}\label{section3}

\begin{proof}
We begin with the following standard lemma. 

\begin{lemma}\label{lemma100}
Let $A\in GL(n,\R).$ 

a) Let $\hat{g}(\xi)=\hat{f}(A\xi), $ for $\xi\in \R^n.$ Then $\|g\|_p=(\det A)^{\frac{1}{p}-1} \|f\|_p,$ for $1\leq p\leq \infty.$

b) Let $\tilde{m}(\xi)= m(A\xi),$ for $\xi\in \R^n.$ Then
$m$ is an $L^p$-multiplier if, and only if $\tilde{m}$ is an
$L^p$-multiplier, $1\leq p\leq \infty.$ The multiplier norms of $m$
and $\tilde{m}$ coincide.
\end{lemma}

We prove Proposition \ref{prop2} a) first. We note that functions $f\in L^p (\R^2)$ for which $\hat{f}(\xi,\mu)=0$ whenever $p(\xi,\mu)=0$ are dense in $L^p (\R^2).$ This and the observation
that $1\leq \sum_k |\psi_k^\lambda(t)|^2,\ t\neq 0,$ allow us to obtain the left-hand-side of the 
desired inequality in the usual way, once we have the right-hand side. (See \cite{Stein2}, Chapter 4.)

We now prove the right-hand-side of the inequality. 
We first recall the definition of $\psi$ in \eqref{eq***} and  define also 
$\Phi_s(t)=\sum_{k=-\infty}^{\infty} r_k(s)  \psi_k(t),$
where $\{r_k\}$ are the Rademacher functions. Then we have
\begin{eqnarray}
|\Phi_s(t)| & \leq &C \label{lp1}\\\
|\Phi_s'(t)| & \leq & \frac{C}{|t|}\label{lp2}\\
|\Phi_s''(t)| & \leq & \frac{C}{|t|^2}\label{lp3},
\end{eqnarray}
with $C$ independent of $s.$ 

We now define $m(\xi,\mu)=\Phi_s(p(\xi,\mu)).$ Then, following
\cite{Stein2} (Chapter 4, Theorem 3), the result would follow if we
could show that 
\begin{eqnarray}
|m(\xi,\mu)| & \leq & C\label{lp*1}\\\
\left|\frac{\partial m}{\partial \xi}(\xi,\mu)\right| & \leq & \frac{C}{|\xi|}\label{lp*2}\\
\left|\frac{\partial m}{\partial \mu}(\xi,\mu)\right| & \leq & \frac{C}{|\mu|}\label{lp*3}\\
\left|\frac{\partial^2 m}{\partial \xi\partial\mu}(\xi,\mu)\right| & \leq & \frac{C}{|\mu\xi|},\label{lp*4}
\end{eqnarray}
again with $C$  independent of $s.$ 

We instead establish and then use inequalities similar to
\eqref{lp*1} to \eqref{lp*4} that are more suited to the geometry at
hand. It is easily seen that there is no $s$ dependence in the following argument, therefore we 
will write $\Phi$ for $\Phi_s.$

We begin by writing $p$ in its factored form:
\begin{eqnarray*}
& & p(\xi,\mu)=c_0(\mu-a_1\xi)^{k_1}(\mu-a_2\xi)^{k_2}\cdots (\mu-a_{d_1}\xi)^{k_{d_1}}\cdot\\
& & \qquad\qquad [(\mu-z_1\xi)(\mu-\overline{z_1}\xi)]^{l_1}\cdots[(\mu-z_{d_2}\xi)(\mu-\overline{z_{d_2}}\xi)]^{l_{d_2}},
\end{eqnarray*}
where $c_0\neq 0, a_i\in \R, i=1,\dots, d_1,$ and $z_j\in {\mathbb C}, j=1,\dots, d_2.$ Without loss of generality, by Lemma \ref{lemma100}b) and scaling,  we may take $c_0=1.$ For each $j=1,\dots,d_2$ we write $z_j=\alpha_j+i\beta_j,$ with
$\alpha_j,\beta_j\in\R.$  Then
$(\mu-z_j\xi)(\mu-\overline{z_j}\xi)= (\mu-\alpha_j\xi)^2+\beta_j^2\xi^2$ and hence
$$p(\xi,\mu)=\prod_{i=1}^{d_1}(\mu-a_i\xi)^{k_i}\prod_{j=1}^{d_2}[(\mu-\alpha_j\xi)^2+\beta_j^2\xi^2]^{l_j}.$$
We take $a_1,\dots,a_{d_1},\alpha_1,\dots,\alpha_{d_2}$ and order them as 
$b_1<b_2<\dots <b_{d_3}.$ We note that $d_3\leq d_1+d_2$ as it is possible that $\alpha_i=a_j$ for some $i,j,$ or $\alpha_{j_1}=\alpha_{j_2}$ for some $j_1,j_2.$  

We now assume that $|\mu|\leq |\xi|$ and choose $\theta\in C^{\infty}(\R)$ such that \newline $\theta(x)=\begin{cases} 1 & x\leq 1\\ 0 & x\geq 2 \end{cases}$ 
and 
\begin{eqnarray*}
 1&= & \theta\left(\left(\frac{\mu}{\xi}-b_1\right)\frac{3}{b_2-b_1}\right)\\
 & & +
\sum_{k=2}^{d_3 -1}\theta\left(\left(b_k-\frac{\mu}{\xi}\right)\frac{3}{b_k-b_{k-1}}\right)\theta\left(\left(\frac{\mu}{\xi}-b_k\right)\frac{3}{b_{k+1}-b_k}\right) \\
& & +\theta\left(\left(b_{d_3}-\frac{\mu}{\xi}\right)\frac{3}{b_{d_3}-b_{d_3 -1}}\right).
\end{eqnarray*}

Of course, when $|\xi|\leq |\mu|,$ an analogous partition of unity with $\frac{\xi}{\mu}$ must be used.

We let $\rho_k=\frac{b_k+b_{k-1}}{2}$ and note that 
\begin{eqnarray*}
& & \theta\left(\left(b_k-\frac{\mu}{\xi}\right)\frac{3}{b_k-b_{k-1}}\right)\theta\left(\left(\frac{\mu}{\xi}-b_k\right)\frac{3}{b_{k+1}-b_k}\right) = 1 \\
& \Longleftrightarrow & \rho_k +\frac{b_k-b_{k-1}}{6}\leq \frac{\mu}{\xi}\leq \rho_{k+1}-\frac{b_{k+1}-b_k}{6}.
\end{eqnarray*}
Also, 
$$\frac{\mu}{\xi}\leq \rho_k-\frac{b_k-b_{k-1}}{6}\Longrightarrow \theta\left(\left(b_k-\frac{\mu}{\xi}\right)\frac{3}{b_k-b_{k-1}}\right)=0$$
and 
$$\frac{\mu}{\xi}\geq \rho_{k+1}+\frac{b_{k+1}-b_k}{6}\Longrightarrow \theta\left(\left(\frac{\mu}{\xi}-b_k\right)\frac{3}{b_{k+1}-b_k}\right)=0.$$

Now we can write 
\begin{eqnarray*}
m(\xi,\mu) & = & \Phi(p(\xi,\mu))\\
& = & \left[\theta\left(\left(\frac{\mu}{\xi}-b_1\right)\frac{3}{b_2-b_1}\right)\right.\\
 & & +
\sum_{k=2}^{d_3 -1}\theta\left(\left(b_k-\frac{\mu}{\xi}\right)\frac{3}{b_k-b_{k-1}}\right)\theta\left(\left(\frac{\mu}{\xi}-b_k\right)\frac{3}{b_{k+1}-b_k}\right) \\
& & \left.+\theta\left(\left(b_{d_3}-\frac{\mu}{\xi}\right)\frac{3}{b_{d_3}-b_{d_3 -1}}\right)\right]\Phi(p(\xi,\mu)).
\end{eqnarray*}

To show that $m$ is an $L^p-$multiplier, it suffices to show that 
each of 
\begin{eqnarray*}
m_1(\xi,\mu) & = & \theta\left(\left(\frac{\mu}{\xi}-b_1\right)\frac{3}{b_2-b_1}\right)\Phi(p(\xi,\mu))\\
m_k(\xi,\mu) & = & \theta\left(\left(b_k-\frac{\mu}{\xi}\right)\frac{3}{b_k-b_{k-1}}\right)\theta\left(\left(\frac{\mu}{\xi}-b_k\right)\frac{3}{b_{k+1}-b_k}\right) \Phi(p(\xi,\mu)),\\
& & \qquad\qquad\qquad \qquad \qquad  \qquad\qquad\qquad\qquad 2\leq k\leq d_3 -1\\
\quad m_{d_3}(\xi,\mu) & = & \theta\left(\left(b_{d_3}-\frac{\mu}{\xi}\right)\frac{3}{b_{d_3}-b_{d_3 -1}}\right)\Phi(p(\xi,\mu))
\end{eqnarray*}
is an $L^p$-multiplier. We will show this for $m_k, 2\leq k\leq d_3 -1.$ The other cases ($m_1$ and $m_{d_3}$) are simpler.
 
By Lemma \ref{lemma100} it suffices to show that $\tilde{m}_k$ is an $L^p$- multiplier, where
$\tilde{m}_k(\xi,\mu)=m_k(A(\xi,\mu))$ for some $A\in GL(2,\R)$ with $\det A=1.$ We take
$A=\left( \begin{array}{cc} 1 & 0 \\ b_k & 1 \end{array}\right).$ 
Then $A(\xi,\mu)=(\xi,\mu+b_k\xi)$ and 
$$\tilde{m}_k(\xi,\mu)=\theta\left(-\frac{\mu}{\xi}\frac{3}{b_k-b_{k-1}}\right)\theta\left(\frac{\mu}{\xi}\frac{3}{b_{k+1}-b_{k}}\right)\Phi(p(\xi,\mu+b_k\xi)).$$

We now have to consider cases according to whether $b_k$ is a real root or the real part of an imaginary root  of $p$ (or possibly both). We first take the case where $b_k=a_{i_0}$ for some $i_0$ and we assume that $a_{i_0}\neq \alpha_j$ for any $j.$   

So 
\begin{eqnarray*}
p(\xi,\mu+b_k\xi) & = & p(\xi,\mu+a_{i_0}\xi)\\
& = & \mu^{k_{i_0}}\prod_{i=1, i\neq i_0}^{d_1}(\mu+(a_{i_0}-a_i)\xi)^{k_i} \prod_{j=1}^{d_2}[(\mu+(a_{i_0}-\alpha_j)\xi)^2+\beta_j^2\xi^2]^{l_j}\\
& := & q(\xi,\mu).
\end{eqnarray*}
Now, 
\begin{equation}\label{57}
\tilde{m}_k(\xi,\mu)\neq 0\Longrightarrow 
-\frac23(b_k-b_{k-1})\leq \frac{\mu}{\xi}\leq \frac23(b_{k+1}-b_k).
\end{equation} 
We claim that, on the support of $\tilde{m}_k,$  
\begin{equation}\label{eq*1}
\left|\frac{\mu}{\xi}+(a_{i_0}-a_i)\right|\geq \begin{cases}\frac13 |a_{i_0}-a_i| & \forall i\neq i_0\\
\frac12\left|\frac{\mu}{\xi}\right|  &\forall i\neq i_0. \end{cases}
\end{equation}
To see this we consider first the case $\frac{\mu}{\xi}\geq 0.$ If also $a_{i_0}-a_i\geq 0$ then the result is clear. If $a_{i_0}-a_i\leq 0$ then we use \eqref{57} and the observation that $a_i-a_{i_0}\geq b_{k+1}-b_k$ to obtain
\begin{eqnarray*}
\left|\frac{\mu}{\xi}+(a_{i_0}-a_i)\right | & \geq &  (a_i-a_{i_0})-\frac{\mu}{\xi}\\
& \geq & (a_{i}-a_{i_0})-\frac{2}{3}(b_{k+1}-b_k)\\
& \geq & \frac13(a_{i_0}-a_i)\\
& \geq &  \frac13(b_{k+1}-b_k)\geq \frac12 \left|\frac{\mu}{\xi}\right|.
\end{eqnarray*}
Next we consider the case $\frac{\mu}{\xi}\leq 0.$ If also $a_{i_0}-a_i\leq 0$ then the result is clear.
If $a_{i_0}-a_i\geq 0$ then we have
\begin{eqnarray*}
\left|\frac{\mu}{\xi}+(a_{i_0}-a_i)\right| & \geq & (a_{i_0}-a_i)+\frac{\mu}{\xi}\\
& \geq & (a_{i_0}-a_i)-\frac23(b_k-b_{k-1})\\
& \geq & \frac13(a_{i_0}-a_i)\\
& \geq & \frac13(b_k-b_{k-1})\geq \frac12\left|\frac{\mu}{\xi}\right|.
\end{eqnarray*}
In a similar way we can show
\begin{equation}\label{eq*2}
\left|\frac{\mu}{\xi}+(a_{i_0}-\alpha_j)\right|\geq \begin{cases}\frac13|a_{i_0}-\alpha_j| & \forall j\\ \frac12\left|\frac{\mu}{\xi}\right| & \forall j\end{cases}.
\end{equation} 

We will now show that $\tilde{m}_k$ is an $L^p$-multiplier by showing that it satisfies
\eqref{lp*1}-\eqref{lp*4}. The first inequality, \eqref{lp*1}, is immediate from \eqref{lp1}. 
We now turn to \eqref{lp*2}. We have
\begin{eqnarray*}
\frac{\partial \tilde{m_k}}{\partial \xi}(\xi,\mu) & = & \theta\left(-\frac{\mu}{\xi}\frac{3}{b_k-b_{k-1}}\right)\theta\left(\frac{\mu}{\xi}\frac{3}{b_{k+1}-b_k}\right)\frac{\partial q}{\partial\xi}(\xi,\mu)\Phi'(q(\xi,\mu))\\
& & +\frac{\mu}{\xi^2}\frac{3}{b_k-b_{k-1}} \theta'\left(-\frac{\mu}{\xi}\frac{3}{b_k-b_{k-1}}\right)\theta\left(\frac{\mu}{\xi}\frac{3}{b_{k+1}-b_k}\right)\Phi(q(\xi,\mu))\\
& &  -\frac{\mu}{\xi^2}\frac{3}{b_{k+1}-b_k}\theta\left(-\frac{\mu}{\xi}\frac{3}{b_k-b_{k-1}}\right)\theta'\left(\frac{\mu}{\xi}\frac{3}{b_{k+1}-b_k}\right)\Phi(q(\xi,\mu))\\
& = & I+II+III.
\end{eqnarray*}
Now $|II|\leq \frac{1}{|\xi|}\left|\frac{\mu}{\xi}\right|\frac{3}{b_{k}-b_{k-1}}\frac{1}{\left|\frac{\mu}{\xi}\frac{3}{b_k-b_{k-1}}\right|} =\frac{1}{|\xi|};$ similarly for $|III|.$ 

For $I$ we have, using \eqref{lp2},
\begin{equation}\label{dagger}
|I|\leq \frac{C}{|q(\xi,\mu)|}\left|\frac{\partial q}{\partial \xi}(\xi,\mu)\right|\left|\theta\left(-\frac{\mu}{\xi}\frac{3}{b_k-b_{k-1}}\right)\theta\left(\frac{\mu}{\xi}\frac{3}{b_{k+1}-b_k}\right)\right|,
\end{equation}
so we need a bound for $\frac{\partial q}{\partial \xi}.$

Now 
\begin{eqnarray*}
& & \frac{\partial q}{\partial \xi}=\\
&  & q(\xi,\mu)\left\{\sum_{i=1, i\neq i_0}^{d_1} \frac{k_i(a_{i_0}-a_i)}{\mu+(a_{i_0}-a_i)\xi} + \sum_{j=1}^{d_2}2l_j\frac{(\mu+(a_{i_0}-\alpha_j)\xi)(a_{i_0}-\alpha_j)+\beta_j^2\xi}{(\mu+(a_{i_0}-\alpha_j)\xi)^2+\beta_j^2\xi^2}\right\}.
\end{eqnarray*}
Now, using \eqref{eq*1}, on the support of $\tilde{m}_k,$ 
$$\left|\sum_{i=1, i\neq i_0}^{d_1}\frac{k_i(a_{i_0}-a_i)}{\mu+(a_{i_0}-a_i)\xi}\right|\leq \frac{1}{|\xi|}\sum_{i=1,i\neq i_0}^{d_1}|k_i|\frac{|a_{i_0}-a_{i}|}{\left|\frac{\mu}{\xi}+(a_{i_0}-a_i)\right|}\leq \frac{3}{|\xi|}\sum_{i=1}^{d_1}|k_i|\leq \frac{C}{|\xi|}.$$
Similarly, using \eqref{eq*2},
\begin{eqnarray*}
& & \left|\sum_{j=1}^{d_2}2l_j\frac{(\mu+(a_{i_0}-\alpha_j)\xi)(a_{i_0}-\alpha_j)+\beta_j^2\xi}{(\mu+(a_{i_0}-\alpha_j)\xi)^2+\beta_j^2\xi^2}\right|\\
& = & \frac{1}{|\xi|}\left| \sum_{j=1}^{d_2}2l_j\frac{(\frac{\mu}{\xi}+(a_{i_0}-\alpha_j))(a_{i_0}-\alpha_j)+\beta_j^2}{(\frac{\mu}{\xi}+(a_{i_0}-\alpha_j))^2+\beta_j^2}\right|\\
& \leq & \frac{2}{|\xi|}\left\{\sum_{j=1}^{d_2}|l_j|\frac{|a_{i_0}-\alpha_j|}{\left|\frac{\mu}{\xi}+(a_{i_0}-\alpha_j)\right|}+\sum_{j=1}^{d_2}|l_j|\right\}\\
& \leq & \frac{8}{|\xi|}\sum_{j=1}^{d_2}|l_j|\leq \frac{C}{|\xi|}.
\end{eqnarray*}
Hence we have
$$\left|\frac{\partial q}{\partial \xi}(\xi,\mu)\right|\leq \frac{C}{|\xi|}|q(\xi,\mu)|$$ 
when $\theta\left(-\frac{\mu}{\xi}\frac{3}{b_k-b_{k-1}}\right)\theta\left(\frac{\mu}{\xi}\frac{3}{b_{k+1}-b_k}\right)\neq 0.$ By \eqref{dagger} this suffices to bound $|I|.$ Thus we have the required bound for \eqref{lp*2}.

The bound for \eqref{lp*3} is similar; we have
\begin{eqnarray*}
\frac{\partial \tilde{m_k}}{\partial \mu}(\xi,\mu) & = & \theta\left(-\frac{\mu}{\xi}\frac{3}{b_k-b_{k-1}}\right)\theta\left(\frac{\mu}{\xi}\frac{3}{b_{k+1}-b_k}\right)\frac{\partial q}{\partial\mu}(\xi,\mu)\Phi'(q(\xi,\mu))\\
& & -\frac{1}{\xi}\frac{3}{b_k-b_{k-1}} \theta'\left(-\frac{\mu}{\xi}\frac{3}{b_k-b_{k-1}}\right)\theta\left(\frac{\mu}{\xi}\frac{3}{b_{k+1}-b_k}\right)\Phi(q(\xi,\mu))\\
& & + \frac{1}{\xi}\frac{3}{b_{k+1}-b_k}\theta\left(-\frac{\mu}{\xi}\frac{3}{b_k-b_{k-1}}\right)\theta'\left(\frac{\mu}{\xi}\frac{3}{b_{k+1}-b_k}\right)\Phi(q(\xi,\mu))\\
& = & IV+V+VI.
\end{eqnarray*}
Then 
$$|V|\leq \frac{1}{|\xi|}\frac{3}{b_k-b_{k-1}}\frac{1}{\left|\frac{\mu}{\xi}\frac{3}{b_k-b_{k-1}}\right|}=\frac{1}{|\mu|},$$
and the bound for $|VI|$ is similar. 

For $IV$ we need a bound for $\frac{\partial q}{\partial \mu}(\xi,\mu).$ Now
 $$\frac{\partial q}{\partial \mu}(\xi,\mu) =
q(\xi,\mu)\left\{\sum_{i=1}^{d_1}\frac{k_i}{\mu+(a_{i_0}-a_i)\xi}+\sum_{j=1}^{d_2}\frac{2l_j (\mu+(a_{i_0}-\alpha_j)\xi)}{(\mu+(a_{i_0}-\alpha_j)\xi)^2+\beta_j^2\xi^2}\right\}$$
and then
\begin{eqnarray*}
& & \left|\frac{\partial q}{\partial \mu}(\xi,\mu)\right| \\
& \leq & \frac{|q(\xi,\mu)|}{|\xi|}
\left\{\sum_{i=1}^{d_1}\frac{|k_i|}{\left|\frac{\mu}{\xi}+(a_{i_0}-a_i)\right|} +\sum_{j=1}^{d_2}\frac{2l_j\left|\frac{\mu}{\xi}+(a_{j_0}-\alpha_j)\right|}{\left|\left(\frac{\mu}{\xi}+(a_{j_0}-\alpha_j)\right)^2+\beta_j^2\right|}\right\}\\
& \leq & \frac{C}{|\mu|}\sum_{i=1}^{d_1}|k_{i}|+\frac{C}{|\mu|}\sum_{j=1}^{d_2}|l_j|,\ \text{by \eqref{eq*1} and \eqref{eq*2}}\\
& \leq & \frac{C}{|\mu|}|q(\xi,\mu)|.
\end{eqnarray*}
By \eqref{lp2} we have $|IV|\leq \frac{C}{|\mu|},$ which completes the proof of  \eqref{lp*3}.

We are left with \eqref{lp*4}.  The proof of this is similar to that of \eqref{lp*3} and \eqref{lp*2} and we omit the details. 
This completes the proof of a) in the case that $b_k$ is a real root. If $b_k$ is the real part of an imaginary root the argument is similar and we omit the details. There is also a possibility that 
$b_k=a_{i_0}=\alpha_{j_0},$ for some $i_0,j_0$ or that $b_k=\alpha_{j_1}=\alpha_{j_2}$ for some $j_1,j_2,$  but again the argument follows as before. 
  
  It is easily checked that b) follows from the same argument.
  \end{proof}

\vspace{.5cm}

\begin{tabular}{lll}
Anthony Carbery & Carlos Kenig & Sarah Ziesler \\
University of Edinburgh & University of Chicago & University of Chicago\\
Edinburgh & Chicago & Chicago\\
EH9 2BJ & IL 60637 & IL 60637\\
U.K. & U.S.A. & U.S.A.
\end{tabular}


\begin{thebibliography}{99}
  
 \bibitem{AS} F. Abi-Khuzam \& B. Shayya, \textsl{Fourier restriction to convex surfaces of revolution in ${\mathbb R}^3,$} Publ. Mat. {\bf 50} (2006) no. 1, 71-85
   
\bibitem{Bak} J-G. Bak, \textsl{ Restrictions of Fourier
transforms to flat curves in ${\mathbf R}^2,$}  Illinois
J. Math. {\bf 38} (1994) no.2,  327-346 

 \bibitem{BCR} J. Bochnak, M. Coste \& M-F. Roy, \textsl{G\'{e}ometrie alg\'{e}brique r\'{e}elle, } Springer-Verlag (1987)
 
\bibitem{BIT} L. Brandolini, A. Iosevich \& G. Travaglini,
\textsl{Spherical means and the restriction phenomenon,}
J.Fourier Anal. Appl. {\bf 7} (2001) no. 4,  369-372

\bibitem{CZ} A. Carbery \& S. Ziesler, \textsl{Restriction and decay for flat hypersurfaces,}
Publ. Mat. {\bf 46} (2002), 405-434

\bibitem{CKZ2}  A. Carbery, C. Kenig \& S. Ziesler, \textsl{Restriction for flat surfaces of revolution in ${\mathbb R}^3,$} Proc. Amer. Math. Soc. {\bf 135 } (2007) no. 6, 1905-1914

\bibitem{CDMM} M. Cowling, S. Disney, G. Mauceri, \& D. M\H{u}ller,
\textsl{Damping oscillatory integrals,} Invent. Math. 
{\bf 101}  (1990) 237-260 

\bibitem{D} S.W. Drury, \textsl{Degenerate curves and harmonic
analysis,} Math. Proc. Cambridge. Philos. Soc. {\bf 108} (1990)
89-96

\bibitem{H} A. Hatcher, \textsl{Algebraic Topology,} Cambridge University Press (2001)



\bibitem{KPV}  C.E. Kenig, G. Ponce \& L. Vega, \textsl{
Oscillatory integrals and regularity of dispersive
equations,} Indiana Univ. Math. J. {\bf 40} (1991) no. 1,  33-69 

\bibitem{KT} C.E.Kenig \& P.Tomas, \textsl{$L^p$ behavior of certain second order differential operators,} Trans. Amer. Math. Soc. {\bf 262} (1980) 521-531

\bibitem{M} J. Milnor, \textsl{On the Betti numbers of real
  varieties,} Proc. Amer. Math. Soc. {\bf 15} (1964) no.2, 275-280
  
 \bibitem{milnor2} J.Milnor, \textsl{Topology from the Differentiable Viewpoint,} University Press of Virginia, Charlottesville (1965), Princeton University Press (1998) 


\bibitem{Oberlin} D. Oberlin, \textsl{A uniform Fourier restriction
  theorem for surfaces in ${\mathbb R}^3$,} Proc. Amer. Math. Soc. {\bf 132} (2004) 1195-1199

\bibitem{S} P. Sj\"{o}lin, \textsl{Fourier multipliers and
estimates of the Fourier transform of measures carried by
smooth curves in ${\mathbf R}^2,$} Studia Math. {\bf 51}
(1974) 169-182

\bibitem{Simon} L. Simon, \textsl{Lectures on Geometric Measure Theory,} Proceedings of the Centre for Mathematical Analysis, vol 3, Australian National University (1983)

\bibitem{Stein2} E.M. Stein, \textsl{Singular Integrals and Differentiability Properties of Functions,} Princeton Mathematical Series, No. 30, Princeton University Press (1970)


\bibitem{Stein} E.M. Stein, \textsl{ Harmonic Analysis;
real-variable methods, orthogonality and oscillatory
integrals,}  Princeton University Press (1993)


\bibitem{T} P. Tomas, \textsl{ A restriction theorem for the
Fourier transform,}  Bull. Amer. Math. Soc.  {\bf 81} (1975) no. 2,
 477-478 

\end{thebibliography}
\end{document}